\documentclass[pdflatex,sn-mathphys,Numbered]{sn-jnl}%

\usepackage{graphicx}%
\usepackage{multirow}%
\usepackage{amsmath,amssymb,amsfonts}%
\usepackage{amsthm}%
\usepackage{mathrsfs}%
\usepackage[title]{appendix}%
\usepackage{xcolor}%
\usepackage{textcomp}%
\usepackage{manyfoot}%
\usepackage{booktabs}%
\usepackage{listings}%
\usepackage{natbib}

\usepackage{anyfontsize}
\usepackage[ruled,vlined]{algorithm2e}
\usepackage{caption}
\usepackage{enumitem}
\usepackage{mathtools,xfrac}
\usepackage[capitalize]{cleveref}
\usepackage{tikz}

\usepackage{pgfplots}\pgfplotsset{compat=1.9}
\usetikzlibrary{pgfplots.groupplots}
\usetikzlibrary{matrix,positioning}

\crefformat{equation}{(#2#1#3)}
\crefformat{condition}{Condition #2#1#3}
\crefrangeformat{equation}{(#3#1#4)--(#5#2#6)}
\crefrangeformat{condition}{Conditions #3#1#4--#5#2#6}
\crefrangeformat{lemma}{Lemmas #3#1#4--#5#2#6}
\crefrangeformat{algorithm}{Algorithms #3#1#4--#5#2#6}
\crefrangeformat{section}{Sections #3#1#4--#5#2#6}
\crefrangeformat{assumption}{Assumptions #3#1#4--#5#2#6}
\crefformat{assumption}{Assumption #2#1#3}

\def\addlegendimage{\csname pgfplots@addlegendimage\endcsname}
\allowdisplaybreaks[2]

\newcommand{\ircg}{\texttt{IR-CG}\xspace}

\newcommand{\cgbio}{\texttt{CG-BiO}\xspace}

\newcommand{\irpg}{\texttt{IR-PG}\xspace}

\newcommand{\bisg}{\texttt{Bi-SG}\xspace}

\newcommand{\nbblo}{\texttt{NB-BLO}\xspace}
\newcommand{\snblo}{\texttt{SNB-LO}\xspace}

\definecolor{gold}{rgb}{0.85,0.65,0}
\colorlet{dgreen}{green!60!black}

\renewcommand{\theequation}{\mbox{\arabic{equation}}}
\def\beq{\begin{equation}}
\def\eeq{\end{equation}}

\def\fnote#1{\footnote}
\newcommand{\BlBox}{\rule{1.5ex}{1.5ex}}
\newcommand{\epr}{\hfill\BlBox\vspace{0.0cm}\par}

\newcommand{\grad}{\ensuremath{\nabla}}

\newcommand{\bbE}{\mathbb{E}}

\newcommand{\bbR}{\mathbb{R}}

\DeclareMathOperator{\opt}{opt}

\DeclareMathOperator*{\argmin}{arg\,min}
\DeclareMathOperator*{\argmax}{arg\,max}

\DeclareMathOperator{\Proj}{Proj}

\DeclareMathOperator{\Trace}{Trace}
\DeclareMathOperator{\Tr}{Tr}

\def\dim{\mathop{{\rm dim}\,}}

\DeclareMathOperator{\Conv}{Conv}

\DeclarePairedDelimiter\abs{\lvert}{\rvert}%

\makeatletter
\let\oldabs\abs
\def\abs{\@ifstar{\oldabs}{\oldabs*}}
\makeatother

\theoremstyle{thmstyleone}%
\newtheorem{theorem}{Theorem}[section]%
\newtheorem{lemma}[theorem]{Lemma}%
\newtheorem{corollary}[theorem]{Corollary}%

\theoremstyle{thmstyletwo}%
\newtheorem{example}{Example}%
\newtheorem{remark}{Remark}%

\theoremstyle{thmstylethree}%
\newtheorem{assumption}{Assumption}%
\newtheorem{condition}{Condition}[section]%

\begin{document}
	
\title[Projection-Free Convex Bilevel Optimization]{A Projection-Free Method for Solving Convex Bilevel Optimization Problems}

	\author[1,2]{\fnm{Khanh-Hung} \sur{Giang-Tran}}\email{tg452@cornell.edu}
	
	\author[1]{\fnm{Nam} \sur{Ho-Nguyen}}\email{nam.ho-nguyen@sydney.edu.au}
	
	\author[3]{\fnm{Dabeen} \sur{Lee}}\email{dabeenl@kaist.ac.kr}
	
	\affil[1]{\orgdiv{Discipline of Business Analytics}, \orgname{The University of Sydney}}
	\affil[2]{\orgdiv{School of Operations Research and Information Engineering}, \orgname{Cornell University}}
	\affil[3]{\orgdiv{Department of Industrial and Systems Engineering}, \orgname{KAIST}}

\abstract{When faced with multiple minima of an \emph{inner-level} convex optimization problem, the \emph{convex bilevel optimization} problem selects an optimal solution which also minimizes an auxiliary \emph{outer-level} convex objective of interest. Bilevel optimization requires a different approach compared to single-level optimization problems since the set of minimizers for the inner-level objective is not given explicitly. In this paper, we propose a new projection-free conditional gradient method for convex bilevel optimization which requires only a linear optimization oracle over the base domain. We establish $O(t^{-1/2})$ convergence rate guarantees for our method in terms of both inner- and outer-level objectives, and demonstrate how additional assumptions such as quadratic growth and strong convexity result in accelerated rates of up to $O(t^{-1})$ and $O(t^{-2/3})$ for inner- and outer-levels respectively. Lastly, we conduct a numerical study to demonstrate the performance of our method.}


	
	\maketitle

	\section{Introduction}\label{sec:intro}
In this paper, we study the following \emph{convex bilevel optimization} problem:
\begin{equation} \label{general simple bilevel problem}
\begin{aligned}
\min_{x \in X} \quad  f(x)\quad \text{s.t.} \quad  x \in X_{\opt}:=\argmin_{z \in X} g(z).
\end{aligned}
\end{equation}
We refer to $g$ as the \emph{inner-level objective function}, and when $g$ and $X$ are convex, $X_{\opt}$ is also convex. We also assume that $f$, the \emph{outer-level objective function}, is convex, which makes \eqref{general simple bilevel problem} a convex optimization problem. We define the inner- and outer-level optimal values as follows:
\begin{equation}\label{eq:optimal-value-def}
        g_{\opt} := \min_{x \in X} g(x), \quad f_{\opt} := \min_{x \in X_{\opt}} f(x).
\end{equation}

Problem \eqref{general simple bilevel problem} is trivial if $X_{\opt}$ is a singleton, i.e., the inner-level objective function has a unique solution over $X$. However, multiple optimal solutions can arise in many practical applications, and \eqref{general simple bilevel problem} may be used to select a solution satisfying auxiliary desirable properties. For example, %
given an underdetermined least squares problem (with inner-level objective $g(x) = \|Ax-b\|_2^2/2$%
), one may seek to find a minimizer of $g$ that also minimizes $f(x) = \|x\|_2^2/2$. %
The solution to this problem is the so-called least norm solution, which admits a closed form $x^* = A^\dagger b$. However, %
a separate numerical method is often required for other types of functions $f$, for which no such closed form solution exists.
Other applications of %
\eqref{general simple bilevel problem} %
include %
dictionary learning, fair classification \citep[Section 2.1]{JiangEtAl2023}, and ill-conditioned inverse problems \citep[Section 5.2]{BeckSabach2014}. 

There are two primary challenges in solving \eqref{general simple bilevel problem}.
First, we do not have an explicit representation of the optimal set $X_{\opt}$ in general, which prevents us from using some common %
operations in optimization such as projection onto or linear optimization %
over $X_{\opt}$. Instead, %
we alternatively consider the \emph{value function formulation} of \eqref{general simple bilevel problem}:
\begin{equation} \label{recasted simple bilevel}
\begin{aligned}
\min_{x \in X} \quad  f(x)\quad
\text{s.t.} \quad  g(x) \leq g_{\opt}.
\end{aligned}
\end{equation}
The second challenge arises because, by the definition of $g_{\opt}$, there exists no $x \in X$ such that $g(x) < g_{\opt}$, hence \eqref{recasted simple bilevel} does not satisfy %
Slater's condition, which means that the Lagrangian dual of \eqref{recasted simple bilevel} may not be solvable. One may attempt to enforce Slater's %
condition by adding a small $\epsilon_g > 0$ to the right-hand side of the constraint so that we consider $g(x) \leq g_{\opt} + \epsilon_g$ instead, but this approach does not solve the actual problem \eqref{recasted simple bilevel}, and may introduce numerical instability \citep[Appendix D]{JiangEtAl2023}.

\subsection{Related literature} \label{subsec:literature}
Several schemes have been devised to tackle the convex bilevel optimization problem \eqref{general simple bilevel problem}. These methods can be grouped into three categories: %
regularization-based, sublevel set, and sequential averaging methods.

\subsubsection{Regularization approach}
This approach combines the inner- and outer-level objectives via the so-called Tikhonov regularization~\cite{tikhonov1977solutions}, i.e., we optimize $\sigma f(x)+g(x)$, where $\sigma >0$ is referred to as the %
regularization parameter. For the setting where $f(x)=\|x\|_2^2$, the Tikhonov regularization scheme converges to the least norm solution asymptotically. Moreover, under some mild conditions, \citet{FriedlanderEtAl2008} showed that for a sufficiently small $\sigma > 0$, the optimal set of the regularized problem $\argmin_{x \in X} \{\sigma f(x)+g(x)\}$ is the same as that of \eqref{general simple bilevel problem}. \citet{FriedlanderEtAl2008} showed that the existence of such $\sigma > 0$ is equivalent to the solvability of the Lagrangian dual of \eqref{recasted simple bilevel}. %
However, the threshold for $\sigma$ is a priori unknown. As an alternative, if we consider a positive sequence of regularization parameters $\{\sigma_t\}_{t \geq 0}$ converging to $0$ and define $s_t \in \argmin_x \{\sigma_t f(x) + g(x) : x \in X\}$ for each $t \geq 0$, then it is known that any accumulation point of $\{s_t\}_{t \geq 0}$ is a solution of \cref{general simple bilevel problem}. %
On the other hand, this requires solving $\min_{x \in X} \{\sigma_t f(x) + g(x)\}$ for each $t$, which can be expensive depending on how fast the sequence of regularization parameters converges. %

A more efficient strategy is to employ only 
cheap first-order updates each time we update $t$. To the best of our knowledge, this idea dates back to \citet{Cabot2005}, who proposed a proximal point-type algorithm to update %
solutions for the unconstrained case, i.e., $X=\bbR^n$. %
\citet{Cabot2005} showed that the iterates converge asymptotically to the optimal solution set. \citet{DuttaPandit2020} extended the %
framework of \citet{Cabot2005} to %
the case of general closed convex $X$. However, \citep{Cabot2005,DuttaPandit2020} did not provide non-asymptotic convergence rates for their methods.%

Note that a proximal point update %
involves solving a possibly expensive optimization problem. %
As an alternative to proximal point-type methods, \citet{Solodov2006} proposed the iterative regularized projected gradient (\irpg) method, where only a projected gradient step is taken at each iteration $t$. \citet{Solodov2006} proved %
asymptotic convergence for the case when $f$ and $g$ are smooth, under some %
appropriate selection of relevant parameters, but this work did not include a non-asymptotic convergence rate for \irpg. %
When $f$ and $g$ are possibly non-smooth, \citet{HelouSimoes2017} proposed a variation of the $\epsilon$-subgradient method with asymptotic convergence provided that $f$ and $g$ are Lipschitz continuous.%

By choosing the regularization and other relevant parameters appropriately, \citet{AminiEtAl2019} proved that the convergence rate of \irpg for the inner-level objective is $O(t^{-(1/2-b)})$ for any fixed $b \in (0,1/2)$ when $f$ and $g$ are non-smooth, but with the additional requirement is that $X$ is compact and $f$ is strongly convex. %
\citet{KaushikYousefian2021} later refined the analysis and %
removed the strong %
convexity assumption on $f$, and moreover, they proved that \irpg admits 
convergence rates of $O(t^{-b})$ and $O(t^{-(1/2-b)})$ for inner- and outer-level objectives, respectively for any $b\in(0,1/2)$. %
In fact, \citet{KaushikYousefian2021} studied a more general setting than~\eqref{general simple bilevel problem} where $X_{\opt}$ is given by the set of solutions to a monotone variational inequality.
\citet{Malitsky2017} studied a version of Tseng’s accelerated gradient method \citep{Tseng2008} applied to~\eqref{general simple bilevel problem} with a convergence rate of $o(t^{-1})$ for the inner-level objective.

Recently, \citet{ShenLingqing2023} proposed two primal-dual-type algorithms in which $\sigma_t$ is adaptively adjusted. %
One %
algorithm works under some minimal convexity assumptions and converges with a %
rate of $O\left(t^{-1/3}\right)$ for both inner- and outer-level objectives. %
The other %
algorithm %
utilizes more structural information about %
the objective functions such as smoothness and strong convexity and %
converges with a %
rate of $O\left(t^{-1/2}\right)$ for both inner- and outer-level objectives. Nevertheless, the algorithms of \citet{ShenLingqing2023} %
demand a tolerance parameter to be set in advance, and therefore no asymptotic convergence was claimed.

\subsubsection{Sublevel %
set method} \label{subsubsec:sub-level} 

Another strategy is to relax the constraint $x\in X_{\opt}$ in \eqref{general simple bilevel problem} by replacing the set of optimal solutions $X_{\opt} = \{x \in X : g(x) \leq g_{\opt}\}$ with an approximation. %
For instance, the minimal norm gradient (MNG) method \citep{BeckSabach2014} constructs an outer approximation of $X_{\opt}$ with two half-spaces over which it %
minimizes $f$, which is assumed to be strongly convex and smooth. The MNG method converges with %
rate %
$O\left(t^{-1/2}\right)$ for the inner-level objective when $g$ is smooth, but no rates were provided for the outer-level objective.

\citet{JiangEtAl2023} introduced the conditional gradient-based 
bilevel optimization (\cgbio) method which approximates $X_{\opt}$ by replacing $g(x)$ with its %
linear approximation. %
They proved a convergence rate of %
$O(t^{-1})$ for both inner- and outer-level objectives when $f,g$ are smooth and $X$ is compact. Recently, \citet{CaoEtAl2023} extended this to the stochastic setting where $f$ and $g$ are given by %
$f(x) := \bbE_\theta[\tilde{f}(x,\theta)]$ and $g(x) := \bbE_\xi[\tilde{g}(x,\xi)]$. When the distributions of %
$\theta$ and $\xi$ have finite support, \citet{CaoEtAl2023} showed a convergence rate of $\tilde{O}(t^{-1})$ for both objectives. Alternatively, if %
$\theta$ and $\xi$ are sub-Gaussian, the method of \citet{CaoEtAl2023} achieves a convergence rate of %
$O\left(t^{-1/2}\right)$. We remark that the frameworks of %
\citet{JiangEtAl2023} and \citet{CaoEtAl2023} both require a predetermined tolerance parameter %
$\epsilon_g>0$ %
based on which asymptotic $(\epsilon_g/2)$-infeasibility guarantee can be established.

Instead of approximating $X_{\opt}$, \citet{DoronShtern2022} provided an alternative formulation of \eqref{general simple bilevel problem} which relies on sublevel %
sets of the outer-level %
objective $f$. Based on this, they developed a method called the iterative approximation and level-set expansion (ITALEX) method \citep{DoronShtern2022}, which %
performs a proximal gradient or generalized conditional gradient step based on a surrogate $\hat{g}_t$ of $g$ over a sublevel set defined as $X \times \{x\in\mathbb{R}^n:f(x) \leq \alpha_t\}$ at each iteration $t$. The surrogate $\hat{g}_t$ and the sublevel set are then updated. 
\citet{DoronShtern2022} 
showed %
convergence rates of $O(t^{-1})$ and $O\left(t^{-1/2}\right)$ for the inner- and outer-level objectives, respectively, when $g$ is a composite function, and $f$ satisfies an error bound condition. %

\subsubsection{Sequential averaging}

Sequential averaging, as its name suggests, proceeds by taking a weighted average of two mappings computed from the current iterate $x_t$ to deduce the next iterate $x_{t+1}$. %
For instance, the bilevel gradient sequential averaging method (BiG-SAM) \citep{SabachEtAl2017} takes a convex combination between a proximal gradient step for the inner-level objective and a gradient step applied to the outer-level function from the current iterate. \citet{SabachEtAl2017} showed asymptotic convergence for $f$ without a convergence rate, and convergence at rate $O(t^{-1})$ for $g$, when $f$ is smooth and strongly convex, and $g$ is a composite smooth function. \citet{ShehuEtAl2021} presented the inertial bilevel gradient sequential averaging method (iBiG-SAM), a variation of the BiG-SAM with an inertial extrapolation step. Although \citet{ShehuEtAl2021} showed asymptotic convergence of the iBiG-SAM without a convergence rate %
under the same assumptions of the BiG-SAM, several numerical examples that were conducted in the study indicated that the iBiG-SAM outperformed the BiG-SAM in those experiments.

\citet{MerchavSabach2023} proposed the bi-sub-gradient (\bisg) method for the case when $f$ is non-smooth and $g$ is composite smooth. At each iteration, a proximal gradient step for $g$ is applied, followed by a subgradient descent step for $f$. \citet{MerchavSabach2023} showed %
convergence rates of $O(t^{-\alpha})$ and $O(t^{-(1-\alpha)})$ for the inner- and outer-level objectives, respectively, with $\alpha \in (1/2,1)$ when $f$ satisfies a quasi-Lipschitz property. %
Moreover, if we further assume that $f$ is composite smooth with the smooth part being strongly convex as well and the non-smooth part of $g$ is Lipschitz continuous, then \bisg guarantees a convergence rate of $O(e^{-c(\beta/4)t^{1-\alpha}})$ for the outer objective where $c,\beta>0$ are some relevant parameters.

\subsection{Contributions}\label{subsec:contribution}

For certain convex domains $X$, linear optimization oracles can be implemented more efficiently than projection operations. For example, in a low-rank matrix completion problem, when $X$ is a nuclear norm ball that consists of $n\times p$ matrices, projection onto $X$ requires a complete singular value decomposition of an $n\times p$ matrix, whereas a linear optimization oracle  requires solving only a maximum singular value problem.

In light of this, %
we present an %
iterative method for solving \eqref{general simple bilevel problem} that requires only a linear optimization oracle over the convex domain $X$ at each iteration. Our method takes the equivalent formulation \eqref{recasted simple bilevel} with the assumption that $X$ is compact and $f, g$ are smooth convex functions. %
Our contributions are summarized in \Cref{results} and as follows:

\begin{itemize}
    \item In \cref{sec:regularized}, we propose what we call the iteratively regularized conditional gradient (\ircg) method, which %
    takes the regularization approach to solve \eqref{general simple bilevel problem}. Unlike other previous regularization-based approaches, we utilize a novel averaging scheme that is necessitated by %
    the conditional gradient updates. We provide conditions on the regularization parameters %
    to ensure asymptotic convergence, as well as convergence rates of $O\left(t^{-p}\right)$ and $O(t^{-(1-p)})$ for the inner- and outer-level objectives, respectively, for any $p \in (0,1)$.

    \item In \cref{sec:acceleration}, we prove that \ircg achieves faster rates of convergence under some additional assumptions. To be more precise, assuming that $g$ exhibits quadratic growth, \ircg attains $O\left(t^{-\min\{1,2p\}} \right)$ for the inner-level gap and $O\left(t^{-\min\{p,1-p\}}\right)$ for the outer-level gap. In addition, if we also assume that $f$ and $X$ are strongly convex, the rate for the outer-level gap can be further accelerated to $O\left(t^{-\min\{2/3,1-2p/3,2-2p\}}\right)$. We show that the best choice for the parameter $p$ is $1/2$.

    \item In \cref{sec:experiments}, we present computational results for a numerical experiment on the matrix completion problem. We compare the performance of \ircg to that of existing methods including \irpg \citep{Solodov2006}, \bisg \citep{MerchavSabach2023} and \cgbio \citep{JiangEtAl2023}. The numerical results demonstrate that \ircg dominates the existing methods, validating our theoretical results.
\end{itemize}
\begin{table*}[!ht]
\begin{center}
\begin{tabular}{c|c|c|c}
\toprule
{\bf Setting} & {\bf Inner-level}   & {\bf Outer-level} & {\bf Results}\\
\midrule
convex $f,g,X$  & $O\left(t^{-p}\right)$ & $O\left(t^{-(1-p)}\right)$ & \Cref{theorem: a convergence rate for PCG-BiO}\\
\midrule
convex $f,g,X$,   & \multirow{2}{*}{$O\left(t^{-\min\{2p,1\}}\right)$} & \multirow{2}{*}{$O\left(t^{-\min\{p,1-p\}}\right)$} & \multirow{2}{*}{\Cref{thm:accelerated-quadratic-growth}} \\
quadratic growth of $g$ & & \\
\midrule
convex $g$,   & \multirow{3}{*}{$O\left(t^{-\min\{2p,1\}}\right)$} & \multirow{3}{*}{$O\left(t^{-\min\{\frac{2}{3},1-\frac{2p}{3},2-2p\}}\right)$} & \multirow{3}{*}{\Cref{thm:accelerated-strongly-convex}}\\
strongly convex $f,X$, & & \\
quadratic growth of $g$ & & \\
\bottomrule
\end{tabular}
\end{center}
\caption{Summary of our results on convergence rates of \ircg.}\label{results}
\end{table*}

We note that the \cgbio method of \citet{JiangEtAl2023} as well as the projection-free variant of the ITALEX method of \citet{DoronShtern2022} also utilizes linear optimization oracles to solve \eqref{general simple bilevel problem}. At each iteration, our method \ircg only requires linear optimization oracles over the base domain $X$. In contrast, \cgbio requires at each iteration a linear optimization oracle over $X \cap H_t$ where $H_t$ is some half-space, which can be significantly more complicated than a linear optimization oracle over $X$; in Appendix~\ref{sec:appendix-implementation} we examine this further for the matrix completion example. ITALEX requires a linear optimization oracle over sublevel 
sets $X\times \{x\in\mathbb{R}^n : f(x) \leq \alpha_t\}$ %
in addition to one over $X$. While some functions $f$ admit simple linear optimization oracles over their sublevel %
sets, another assumption %
to ensure convergence for ITALEX is that the sublevel %
sets of $f$ are bounded, which is often restrictive in practice. %
In fact, the problem instance for our numerical experiments in \cref{sec:experiments} gives rise to unbounded sublevel sets.

	\section{Iteratively regularized conditional gradient method}\label{sec:regularized}

In this section, we describe an iterative regularization approach to solve problem \eqref{recasted simple bilevel}. 
The main difficulty in solving \eqref{recasted simple bilevel} comes from dealing with the functional constraint $g(x) \leq g_{\opt}$. %
A common way to handle a functional constraint is by considering 
the Lagrangian
$$L(x, \lambda) := f(x)+\lambda (g(x)-g_{\opt}), \quad x \in X, \lambda \geq 0,$$
and the associated Lagrangian dual problem $\sup_{\lambda \geq 0} \min_{x \in X} L(x,\lambda)$. Under \cref{ass:smooth-functions} below, by Sion's minimax theorem \citep[Theorem 3.4]{Maurice1958}, strong duality holds, i.e.,
$$\sup_{\lambda \geq 0} \min_{x \in X} L(x,\lambda) = \min_{x \in X} \sup_{\lambda \geq 0} L(x,\lambda).$$
Note that given $x \in X$, we have
\begin{align*}
    \sup_{\lambda \geq 0} L(x,\lambda) =\begin{cases}
    f(x), & x \in X_{\opt}\\
    +\infty, & x \in X \setminus X_{\opt},
    \end{cases}
\end{align*}
and therefore $\min_{x \in X} \sup_{\lambda \geq 0} L(x,\lambda)$ is equivalent to %
\eqref{recasted simple bilevel}, hence
\begin{equation}\label{minimax}
\sup_{\lambda \geq 0} \min_{x \in X} L(x,\lambda) = f_{\opt}.
\end{equation}

While strong duality holds, the dual is not always \emph{solvable} in general. %
Hence, we do not know whether there exists a finite $\lambda_{\opt} \geq 0$ such that $\min_{x \in X} L(x,\lambda_{\opt}) = f_{\opt}$ and $X_{\opt}  \subseteq  \argmin_{x \in X} L(x,\lambda_{\opt})$ or not. However, since $L(x,\lambda)$ is non-decreasing in $\lambda$ for each $x \in X$, regardless of the existence of $\lambda_{\opt}$, we must have that
\begin{equation*}\label{eq:limit}
	\sup_{\lambda \geq 0} \min_{x \in X} L(x,\lambda) = \lim_{\lambda \to \infty} \left(\min_{x \in X}L(x,\lambda)\right) = f_{\opt}.
\end{equation*}
This suggests that we consider a sequence of multipliers $\lambda_t$ and their associated solutions $x_t \in \argmin_{x \in X} L(x,\lambda_t)$ with $\lambda_t \to \infty$. %
Notice that while $L(x,\lambda)$ contains $g_{\opt}$, which is generally unknown a priori, obtaining $x_t$ does not require knowledge of $g_{\opt}$ at all. That said, obtaining $x_t$ by optimizing $L(x,\lambda_t)$ may still be expensive. \citet{Solodov2006} proposed the \emph{iteratively regularized projected gradient} (\irpg) method which performs a single projected gradient step to obtain $x_{t+1}$, i.e., 
$$x_{t+1} := \Proj_X\left(x_t - \frac{\alpha_t}{\lambda_t} \grad_x L(x_t,\lambda_t) \right) = \Proj_X\left(x_t - \alpha_t \left(\frac{1}{\lambda_t} \grad f(x_t) + \grad g(x_t)\right) \right).$$ 
\citet[Theorem 3.2]{Solodov2006} showed that if $\lambda_t \to \infty$ sufficiently slowly (in the sense that $\sum_{t \geq 0} 
1/\lambda_t = \infty$), then the Euclidean distance between $x_t$ and the solution set of the convex bilevel optimization problem \eqref{general simple bilevel problem} converges to $0$.

Inspired by the result of Solodov, we propose what we call the \emph{iteratively regularized conditional gradient} (\ircg) method, outlined in \cref{alg:IR-CG} below, which essentially replaces the projection step with a conditional gradient-type step. To simplify the analysis, it is convenient to define 
\begin{equation}\label{eq:phi_t}
	\sigma_t:= \frac{1}{\lambda_t}, \quad \Phi_t(x) := \frac{1}{\lambda_t} L(x, \lambda_t)+g_{\opt} = \sigma_t f(x) + g(x).
\end{equation}
The sequence $\{\sigma_t\}_{t \geq 0}$ are referred to as \emph{regularization parameters}. Based on the discussion above, we impose the following condition.

\begin{condition}\label{IRCG-C1} 
	The regularization parameters $\{\sigma_t\}_{t \geq 0}$ are strictly decreasing, positive, and converge to $0$. 
\end{condition}

\begin{algorithm}[htpt] \caption{Iteratively regularized conditional gradient (\ircg) algorithm.} \label{alg:IR-CG}
	\KwData{Parameters $\{\alpha_t\}_{t \geq 0} \subseteq [0,1]$,$ \{\sigma_t\}_{t\geq 0} \subseteq \mathbb{R}_{++}$.}
	\KwResult{sequences $\{x_t,z_t\}_{t \geq 1}$.}
        Initialize $x_0 \in X$\;
        \For{$t=0,1,2,\ldots$}{
        Compute
        \begin{subequations}\label{eq:ircg-iter}
       	\begin{align}
        v_t &\in \argmin_{v \in X}\{(\sigma_t \grad f(x_t) + \grad g(x_t))^\top v\}\\
        x_{t+1} &:= x_t+\alpha_t(v_t-x_t)\\
        S_{t+1} &:= (t+2)(t+1)\sigma_{t+1} + \sum_{i\in [t+1]}(i+1)i(\sigma_{i-1}-\sigma_i)\label{S}\\
        z_{t+1} &:= \frac{(t+2)(t+1)\sigma_{t+1} x_{t+1}+ \sum_{i\in [t+1]}(i+1)i(\sigma_{i-1}-\sigma_i)x_i}{S_{t+1}}.\label{z}
        \end{align}
        \end{subequations}
        \vspace{-10pt}
        }	
\end{algorithm}

\noindent In \cref{alg:IR-CG}, the $z_t$ terms given by~\eqref{z} are taken as a convex combination of $x_1,\ldots,x_t$, which are the usual terms in the conditional gradient algorithm, while the sum of the weights is $S_t$ given by~\eqref{S}. While we give an explicit bound for the inner-level optimality gap of $x_t$ (\cref{lemma:lastiterate}) and show asymptotic convergence of the outer-level optimality gap (\cref{theorem:asymptotic}), by instead using $z_t$, we obtain explicit bounds for both inner- and outer-levels (\cref{theorem: a convergence rate for PCG-BiO}). This specific choice of convex combinations naturally arises in our analysis (particularly  \cref{lemma:recursive-rule} and \cref{lemma: common bound}).

\begin{remark}\label{rem:recursion}
By setting $S_0:=0$ and $z_0:=0$, for $t\geq 1$, $S_t$ and $z_t$ defined in \eqref{S} and \eqref{z} %
can be efficiently computed by the following recursive formulae:
\begin{align*}
	S_{t+1} = S_t+2(t+1) \sigma_t, \quad z_{t+1} = \frac{S_t z_t-(t+1)t\sigma_t x_t+(t+2)(t+1)\sigma_t x_{t+1}}{S_{t+1}}.
\end{align*}
\epr
\end{remark}

Before analyzing \cref{alg:IR-CG}, we introduce the following convexity and smoothness conditions on the structure of \eqref{general simple bilevel problem}, which are standard in conditional gradient methods.

\begin{assumption}\label{ass:smooth-functions}
Let $\| \cdot \|$ be an arbitrary norm on $\mathbb{R}^n$  and $\| \cdot \|_*$ be its dual norm. We consider the following conditions on $f$, $g$, and $X$:
\begin{enumerate}[label=(\alph*), ref=\cref{ass:smooth-functions}(\alph*)]
\item\label{ass:X-compact} $X \subseteq \mathbb{R}^n$ is convex and compact with diameter $D < \infty$, i.e., $\|x-y\|\leq D$ for any $x,y \in X$.
\item \label{ass:g-smooth}$g$ is convex and $L_g$-smooth on an open neighbourhood of $X$, i.e., it is continuously differentiable and its derivative is $L_g$-Lipschitz: $$\|\grad g(x) - \grad g(y)\|_* \leq L_g \|x-y\|,$$ 
for any $x,y \in X$.
\item \label{ass:f-smooth} $f$ is convex and $L_f$-smooth on an open neighbourhood of $X$.
\end{enumerate}
\end{assumption}

Define $\bar{\alpha}_t:=2/(t+2)$ for each $t \geq 0$. We provide the convergence analysis for \cref{alg:IR-CG} under three different step size selection schemes, which are 
\begin{enumerate}
    \item \emph{open loop:}
\begin{equation} \label{eq:openloop}
    \begin{aligned}
        \alpha_t = \bar{\alpha}_t = \frac{2}{t+2}, \quad \forall t \geq 0,
    \end{aligned}
\end{equation}
\item  \emph{closed loop:}
\begin{equation} \label{eq:inexactlinesearch}
    \begin{aligned}
        \alpha_t \in \argmin_{\alpha \in [0,1]} \left\{\alpha \nabla \Phi_t(x_t)^\top (v_t-x_t)+\frac{(\sigma_tL_f+L_g)\alpha^2}{2}\|v_t-x_t\|^2\right\},  \ \forall t \geq 0,
    \end{aligned}
\end{equation}
\item \emph{line search:}
\begin{equation} \label{eq:exactlinesearch}
	\begin{aligned}
		\alpha_t \in \argmin_{\alpha \in [0,1]} \Phi_t(x_t+\alpha(v_t-x_t)),  \quad \forall t \geq 0.
	\end{aligned}
\end{equation}
\end{enumerate}

\begin{lemma} \label{lemma:recursive-rule}
Let $\{x_t\}_{t \geq 0}$ denote the iterates generated by \cref{alg:IR-CG} with step sizes $\{\alpha_t\}_{t\geq 0}$ chosen via \eqref{eq:inexactlinesearch} or \eqref{eq:exactlinesearch}. If \cref{ass:smooth-functions} and \cref{IRCG-C1} hold, then for each $t \geq 0$ and any $\alpha \in [0,1]$,
\begin{align}\label{ineq1:lemma:recursive-rule}
    \begin{aligned}
        &\Phi_t(x_{t+1})-\Phi_t(x_{\opt}) \\
        &\leq \Phi_t(x_t)-\Phi_t(x_{\opt}) + \alpha \nabla \Phi_t(x_t)^\top (v_t-x_t)+\frac{(\sigma_t L_f+L_g)\alpha^2}{2}\|v_t-x_t\|^2.
    \end{aligned}
\end{align}
Furthermore, when $\{\alpha_t\}_{t\geq 0}$ are chosen via any one of \cref{eq:openloop,eq:exactlinesearch,eq:inexactlinesearch}, for each $t \geq 0$,
    \begin{equation}\label{ineq2:lemma:recursive-rule}
    \Phi_t(x_{t+1})-\Phi_t(x_{\opt}) \leq \left(1-\frac{2}{t+2}\right)(\Phi_t(x_t)-\Phi_t(x_{\opt}))+\frac{2(L_f \sigma_t+L_g)D^2}{(t+2)^2}.
    \end{equation}
\end{lemma}
\begin{proof}[Proof of \cref{lemma:recursive-rule}]
Under the step size selection scheme %
\eqref{eq:exactlinesearch}, we have
\begin{align*}
    \Phi_t(x_{t+1})
    &\leq \Phi_t\left(x_t+\alpha(v_t-x_t)\right)\\
    &\leq \Phi_t(x_t)+\alpha \nabla \Phi_t(x_t)^\top (v_t-x_t)+\frac{(\sigma_tL_f+L_g)\alpha^2}{2}\|v_t-x_t\|^2
\end{align*}
where the first inequality follows from the choice of $\alpha_t$ and the second inequality holds because $\Phi_t$ is $(\sigma_tL_f + L_g)$-smooth.
Under the step size schedule %
\eqref{eq:inexactlinesearch}, we have
\begin{align*}
    \Phi_t(x_{t+1})
    &\leq \Phi_t(x_t)+\alpha_t \nabla \Phi_t(x_t)^\top (v_t-x_t)+\frac{(\sigma_tL_f+L_g)\alpha_t^2}{2}\|v_t-x_t\|^2\\
    &\leq \Phi_t(x_t)+\alpha \nabla \Phi_t(x_t)^\top (v_t-x_t)+\frac{(\sigma_tL_f+L_g)\alpha^2}{2}\|v_t-x_t\|^2
\end{align*}
where the first inequality holds because $\Phi_t$ is $(\sigma_tL_f + L_g)$-smooth and the second inequality is due to the choice of $\alpha_t$. Hence, we deduce that~\eqref{ineq1:lemma:recursive-rule} holds under~\eqref{eq:inexactlinesearch} and \eqref{eq:exactlinesearch}. In particular, it follows from~\eqref{ineq1:lemma:recursive-rule} with $\alpha = \bar{\alpha}_t$ that
$$\Phi_t(x_{t+1}) \leq  \Phi_t(x_t)+\bar{\alpha}_t \nabla \Phi_t(x_t)^\top (v_t-x_t)+\frac{(\sigma_tL_f+L_g)\bar{\alpha}_t^2}{2}\|v_t-x_t\|^2$$
under \cref{eq:exactlinesearch} or \cref{eq:inexactlinesearch}. In fact, this inequality also holds under~\eqref{eq:openloop} since $\alpha_t=\bar{\alpha}_t$. Moreover, since $f$ and $g$ are convex, $\Phi_t$ is convex. Thus, through the definition of $v_t$ and convexity of $\Phi_t$ we have that
$$\nabla \Phi_t(x_t)^\top (v_t-x_t) \leq \nabla \Phi_t(x_t)^\top (x_{\opt}-x_t) \leq \Phi_t(x_{\opt})-\Phi_t(x_t), \quad \forall t \geq 0.$$
Combining the above observations with $\|v_t-x_t\| \leq D$, we obtain \eqref{ineq2:lemma:recursive-rule}, as required.
\end{proof}

\begin{corollary}\label{lemma: common bound}
Let $\{x_t\}_{t \geq 0}$ denote the iterates generated by \cref{alg:IR-CG} with step sizes $\{\alpha_t\}_{t \geq 0}$ chosen via any one of \cref{eq:openloop,eq:exactlinesearch,eq:inexactlinesearch}. If \cref{ass:smooth-functions} and \cref{IRCG-C1} hold, then for each $t \geq 1$, it holds that
\begin{equation} \label{eq:weighted-sum}
\begin{aligned}
&(t+1)t\left(\Phi_t(x_t)-\Phi_t(x_{\opt})\right)+ \sum_{i \in [t]}(i+1)i(\sigma_{i-1}-\sigma_i)(f(x_i)-f_{\opt})\\ 
&\leq 2(L_f \sigma_0+L_g)D^2t.
\end{aligned}
\end{equation}
Consequently, we have for any $t \geq 1$,
\begin{align}
&(t+1)t\sigma_t(f(x_t)-f_{\opt})+\sum_{i \in [t]}(i+1)i(\sigma_{i-1}-\sigma_i)(f(x_i)-f_{\opt})\notag\\
&\leq 2(L_f \sigma_0+L_g)D^2t
\label{eq:weighted-sum-f},\\
&f(z_t)-f_{\opt} \leq \frac{2(L_f \sigma_0+L_g)D^2t}{S_t}.\label{eq:bound-fz}
\end{align}
\end{corollary}
\begin{proof}[Proof of \cref{lemma: common bound}]
Note that
\begin{align*}
   & \Phi_t(x_{t+1})-\Phi_t(x_{\opt})\\
   &=\sigma_{t+1}(f(x_{t+1})-f_{\opt}) +g(x_{t+1})-g_{\opt} +(\sigma_{t}-\sigma_{t+1})(f(x_{t+1})-f_{\opt})\\
    &= \Phi_{t+1}(x_{t+1})-\Phi_{t+1}(x_{\opt})+(\sigma_{t}-\sigma_{t+1})(f(x_{t+1})-f_{\opt}).
\end{align*}
Then, if we select $\alpha_t$ based on \cref{eq:openloop,eq:exactlinesearch,eq:inexactlinesearch}, it follows from \cref{ineq2:lemma:recursive-rule} in \cref{lemma:recursive-rule} that
\begin{equation} \label{eq:recursive-rule}
    h_{t+1} \leq \left(1-\frac{2}{t+2}\right)h_t + \eta_t, \quad \forall
    t \geq 0,
\end{equation}
where for each $t \geq 0$, we define the notation
\begin{align*}
    h_t := \Phi_t(x_t)-\Phi_t(x_{\opt}), \quad\eta_t := \frac{2(L_f \sigma_t+L_g)D^2}{(t+2)^2} -(\sigma_t-\sigma_{t+1})(f(x_{t+1})-f_{\opt}).
\end{align*}
From \eqref{eq:recursive-rule}, for any $t \geq 1$, we obtain
$$(t+1)t h_t \leq (t+1)t\left(\frac{t-1}{t+1} h_{t-1} + \eta_{t-1}\right) = t(t-1)h_t+(t+1)t\eta_{t-1}.$$
Thus, for $t \geq 1$, it holds that
\begin{align*}
    (t+1)t h_t %
    =\sum_{i \in [t]} \left((i+1)i h_i-i(i-1) h_{i-1}\right)\leq \sum_{i \in [t]} (i+1)i\eta_{i-1}.
\end{align*}
By the fact that $i/(i+1)<1$ for $i \in [t]$, we have
\begin{align*}
    &(t+1)t\left(\Phi_t(x_t)-\Phi_t(x_{\opt})\right)\\
    &\leq \sum_{i\in [t]}\left(2(\sigma_{i-1} L_f+L_g)D^2 - (i+1)i(\sigma_{i-1}-\sigma_i)(f(x_i)-f_{\opt})\right).
\end{align*}
Since $\sigma_{i-1}\leq\sigma_0$ for $ i \in [t]$ by \cref{IRCG-C1}, we obtain $$ \sum_{i\in [t]}2(\sigma_{i-1} L_f+L_g)D^2 \leq \sum_{i\in [t]}2(\sigma_0 L_f+L_g)D^2 = 2(\sigma_0 L_f+L_g)D^2t,$$
which implies \eqref{eq:weighted-sum}. Since $g(x_T) \geq g_{\opt}$ and $\Phi_t(x)=\sigma_t f(x) + g(x)$, \cref{eq:weighted-sum-f} follows from \cref{lemma: common bound}.
Then \cref{eq:bound-fz} follows from convexity of $f$, the definition of $z_t$, and \eqref{eq:weighted-sum-f}.
\end{proof}

To establish the desired convergence results for \cref{alg:IR-CG}, we need to impose further conditions on the regularization parameters $\{\sigma_t\}_{t \geq 0}$ which are stated below.

\begin{condition} \label{IRCG-C2}
    There exists $t_0 \in \mathbb{N}$ such that if $t \geq t_0$, we have $(t+2) \sigma_{t+1} > (t+1)\sigma_t$, and $(t+1)\sigma_t \to \infty$ as $t \to \infty$.  
\end{condition}

\begin{condition} \label{IRCG-C3}
There exists $L \in \mathbb{R}$ such that
$$\lim_{t \to \infty} t\left(\frac{\sigma_t}{\sigma_{t+1}}-1\right)=L.$$
\end{condition}

The following parameters will appear in our analysis and convergence rates.
\begin{subequations}\label{eq:CV}
\begin{align}
C&:=\left(f_{\opt}-\min_{x \in X} f(x)\right)\left(1+\sup_{t \geq 0} \frac{\sum_{i \in [t]}(i+1)i(\sigma_{i-1}-\sigma_i)}{(t+1)t\sigma_t}\right)+\frac{2(\sigma_0 L_f+L_g)D^2}{\min_{t \geq 0} (t+1)\sigma_t}\label{C}\\
V&:= \sup_{t \geq 0} \frac{\sum_{i \in [t]}(i+1)i(\sigma_{i-1}-\sigma_i)\sigma_i}{(t+1)t\sigma_t^2}.\label{V}
\end{align}
\end{subequations}
It is not obvious that $C,V$ as defined are finite. The following lemma states that any choice of parameters satisfying \crefrange{IRCG-C1}{IRCG-C3} will give finite $C$ and $V$.

\begin{lemma} \label{lemma:CV}
If \cref{ass:smooth-functions} and \crefrange{IRCG-C1}{IRCG-C3} hold, then $C$ defined in \eqref{C} is finite. Furthermore, if $L$ defined in \cref{IRCG-C3} is less than $1$, then $V$ defined in \eqref{V} is finite.
\end{lemma}
\begin{proof}[Proof of \cref{lemma:CV}]
We note that under Conditions \ref{IRCG-C1} and \ref{IRCG-C2}, $0 \leq L \leq 1$. While it is trivial to see $L\geq 0$, $L\leq1$ needs some justifications. For the sake of contradiction, if $L>1$, then for some sufficiently large $t$, we have
$$t\left(\frac{\sigma_t}{\sigma_{t+1}}-1\right) > 1 \iff \frac{\sigma_t}{\sigma_{t+1}} > \frac{t+1}{t} \iff (t+1)\sigma_{t+1} < t\sigma_t,$$
 which contradicts \cref{IRCG-C2}.
We observe that
$$
\lim_{t \to \infty} \frac{(t+1)t(\sigma_{t-1}-\sigma_t)}{(t+1)t\sigma_t-t(t-1)\sigma_{t-1}}
		= \lim_{t \to \infty} \frac{(1+\frac{2}{t-1})(t-1)(\frac{\sigma_{t-1}}{\sigma_t}-1)}{2-(t-1)(\frac{\sigma_{t-1}}{\sigma_t}-1)}= \frac{L}{2-L}
$$
where the first equality is obtained by diving both the numerator and the denominator by $t\sigma_t$ and the second equality comes from~\cref{IRCG-C3}. Moreover,
\begin{align*}
&\lim_{t \to \infty} \frac{(t+1)t(\sigma_{t-1}-\sigma_t)}{(t+1)t\sigma_t-t(t-1)\sigma_{t-1}}\\
&=\lim_{t \to \infty} \frac{\sum_{i\in[t]}(i+1)i(\sigma_{i-1}-\sigma_i)-\sum_{i\in[t-1]}(i+1)i(\sigma_{i-1}-\sigma_i)}{(t+1)t\sigma_t-t(t-1)\sigma_{t-1}}\\
&=\lim_{t \to \infty} \frac{\sum_{i \in [t]}(i+1)i(\sigma_{i-1}-\sigma_i)}{(t+1)t\sigma_t}
\end{align*}
where the second equality follows from  the Stolz–Ces\`{a}ro theorem. (This states that for two sequences of real numbers $\{u_n\}_{n\geq 1}$ and $\{v_n\}_{n\geq 1}$ where $\{v_n\}_{n \geq 1}$ is strictly monotone and divergent, if $\lim_{n \to \infty} (u_{n+1}-u_n)/(v_{n+1}-v_n) = l$ for $l \in \mathbb{R}\cup \{\pm \infty\}$, then $\lim_{n \to \infty} u_n/v_n = l$.)
This implies that
\begin{equation} \label{eq:L/(2-L)}
		\lim_{t \to \infty} \frac{\sum_{i \in [t]}(i+1)i(\sigma_{i-1}-\sigma_i)}{(t+1)t\sigma_t} = \frac{L}{2-L},
\end{equation}
and therefore
$$\sup_{t \geq 0} \frac{\sum_{i \in [t]}(i+1)i(\sigma_{i-1}-\sigma_i)}{(t+1)t\sigma_t} \in (0, \infty).$$
Using \cref{IRCG-C2}, we have $\min_{t \geq 0} (t+1)\sigma_t > 0$. Thus, $C$ is finite.

From
Conditions \ref{IRCG-C1} and \ref{IRCG-C2}, $\{t\sigma_t\}_{t \geq 0}$ is eventually increasing since $t\sigma_t = (t+1)\sigma_t-\sigma_t$. This also implies that $\{(t+1)t\sigma_t^2\}_{t \geq 0}$ is eventually increasing and diverges to $\infty$ because $(t+1)t\sigma_t^2=(t+1)\sigma_t\cdot t\sigma_t$. Note that
\begin{align*}
    &\lim_{t \to \infty} \frac{\sum_{i \in [t]}(i+1)i(\sigma_{i-1}-\sigma_i)\sigma_i}{(t+1)t\sigma_t^2}\\
    &=\lim_{t \to \infty} \frac{(t+2)(t+1)(\sigma_t-\sigma_{t+1})\sigma_{t+1}}{(t+2)(t+1)\sigma_{t+1}^2-(t+1)t\sigma_t^2}\\
    & = \lim_{t \to \infty} \frac{\left(\frac{t+2}{t}\right)t\left(\frac{\sigma_t}{\sigma_{t+1}}-1\right)}{\left(2-t\left(\frac{\sigma_t}{\sigma_{t+1}}-1\right)\right)\left(\frac{\sigma_t}{\sigma_{t+1}}+1\right)-2\frac{\sigma_t}{\sigma_{t+1}}} = \frac{L}{(2-L)2-2} = \frac{L}{2(1-L)}
\end{align*}
where the first equality is implied by the Stolz-Ces\`aro theorem, the second equality is obtained from dividing both the denominator and the numerator by $(t+1)\sigma_{t+1}^2$, and the third equality comes from~\cref{IRCG-C3} and 
$$\lim_{t\to\infty}\frac{\sigma_t}{\sigma_{t+1}} = 1+ \lim_{t\to\infty}\frac{1}{t}\cdot t\left(\frac{\sigma_t}{\sigma_{t+1}}-1\right)=1.$$
Note that $L/2(1-L)\in[0,\infty)$ since $L \in [0,1)$. Thus, $V$ is finite.
\end{proof}
We now provide an $o(1)$ upper bound on $g(x_t)-g_{\opt}$, which is then used in \cref{lemma:twoupperbounds} to bound the gap $g(z_t) - g_{\opt}$.
\begin{lemma} \label{lemma:lastiterate}
Let $\{x_t\}_{t \geq 0}$ denote the iterates generated by \cref{alg:IR-CG} with step sizes $\{\alpha_t\}_{t \geq 0}$ chosen via any one of \cref{eq:openloop,eq:exactlinesearch,eq:inexactlinesearch}, and let $C$ be defined as in~\eqref{C}. If \cref{ass:smooth-functions} and \crefrange{IRCG-C1}{IRCG-C3} hold, then for each $t \geq 1$, it holds that
\begin{equation} \label{eq:bound-inner}
\begin{aligned}
    g(x_t)-g_{\opt} \leq C\sigma_t.
\end{aligned}
\end{equation}
\end{lemma}
\begin{proof}[Proof of \cref{lemma:lastiterate}]
First, we define $F:= f_{\opt}-\min_{x \in X} f(x) \geq 0$. 
Then we have $F\geq f_{\opt}-f(x_t)$ for $t \geq 0$. This, together with \cref{IRCG-C1} and \eqref{eq:weighted-sum} from \cref{lemma: common bound}, implies that
\begin{align*}
    \frac{g(x_t)-g_{\opt}}{\sigma_t}
    &\leq \frac{1}{(t+1)t\sigma_t} \left( 2(\sigma_0 L_f+L_g)D^2t+F\sum_{i \in [t]} (i+1)i(\sigma_{i-1}-\sigma_i)\right)+F\\
    &= \frac{2(\sigma_0 L_f+L_g)D^2}{(t+1)\sigma_t}
+F\left(1+\frac{\sum_{i \in [t]} (i+1)i(\sigma_{i-1}-\sigma_i)}{(t+1)t\sigma_t}\right).
\end{align*}
Here, the right-hand side is at most $C$, and therefore, $g(x_t)-g_{\opt}\leq C\sigma_t$.
\end{proof}

We now show asymptotic convergence of $\{x_t\}_{t\geq 0}$ for the outer-level objective $f$. %
\begin{theorem}\label{theorem:asymptotic}
Under the same conditions as \cref{lemma:lastiterate}, we have
\[\lim_{t \to \infty} f(x_t) = f_{\opt}, \quad \lim_{t \to \infty} g(x_t) = g_{\opt}.\]
\end{theorem}
\begin{proof}[Proof of \cref{theorem:asymptotic}]
From \cref{lemma:lastiterate}, we have $\lim_{t \to \infty} g(x_t)=g_{\opt}$, which implies that any limit point of $\{x_t\}_{t \geq 0}$ should be in $X_{\opt}$. Furthermore, since $X$ is compact, there will never be a subsequence of $\{x_t\}_{t \geq 0}$ that diverges to infinity. Thus $\liminf_{t \to \infty} f(x_t) \geq f_{\opt}$, i.e.,
for any $\epsilon > 0$, there exists $t_\epsilon \geq 1$ such that
$f(x_t)-f_{\opt} \geq -\epsilon$ for any $t > t_\epsilon$.
For $t > t_\epsilon$, we decompose the sum in 
\eqref{eq:weighted-sum-f} into separate terms as follows:
\begin{align*}
    f(x_t)-f_{\opt}
    &\leq \frac{2(L_f \sigma_0+L_g)D^2 t + \sum_{i \in [t_\epsilon]}(i+1)i(\sigma_{i-1}-\sigma_i)(f_{\opt}-f(x_i))}{(t+1)t\sigma_t}\\
    &\quad + \frac{ \epsilon\left(\sum_{i \in [t]}(i+1)i(\sigma_{i-1}-\sigma_i)-\sum_{i \in [t_\epsilon]}(i+1)i(\sigma_{i-1}-\sigma_i)\right)}{(t+1)t\sigma_t}.
\end{align*}
Then it follows that %
\begin{align*}
\limsup_{t \to \infty} (f(x_t)-f_{\opt}) &\leq \epsilon\cdot \limsup_{t \to \infty} 	\frac{\sum_{i \in [t]}(i+1)i(\sigma_{i-1}-\sigma_i)}{(t+1)t\sigma_t}  = \frac{\epsilon L}{2-L}
\end{align*}
where the equality is from \eqref{eq:L/(2-L)}. Thus $\limsup_{t\to \infty} f(x_t) \leq f_{\opt}$
as $\epsilon$ can be chosen arbitrarily small, and hence $\lim_{t \to \infty} f(x_t) = f_{\opt}$.
\end{proof}

Turning attention to the averaged sequence $\{z_t\}_{t \geq 1}$, we provide concrete bounds on the inner- and outer-level optimality gaps.
\begin{lemma} \label{lemma:twoupperbounds}
Let $\{z_t\}_{t \geq 1}$ denote the iterates generated by \cref{alg:IR-CG} defined in \cref{z}.
Under the same conditions as \crefrange{lemma:CV}{lemma:lastiterate}, and with $C,V$ defined in \cref{eq:CV}, for any $t \geq 1$ we have
\begin{align*}
    f(z_t) - f_{\opt} &\leq \frac{2(L_f\sigma_0+L_g)D^2}{(t+1)\sigma_t},\\
    g(z_t)-g_{\opt} &\leq C(1+V) \sigma_t.
\end{align*}
\end{lemma}
\begin{proof} [Proof of \cref{lemma:twoupperbounds}]
Since $S_t \geq (t+1)t\sigma_t$ (see \cref{rem:recursion}) and \eqref{eq:bound-fz}, we have that
$$f(z_t)-f_{\opt} \leq \frac{2(L_f\sigma_0+L_g)D^2 t}{S_t}\leq \frac{2(L_f\sigma_0+L_g)D^2}{(t+1)\sigma_t}.$$
Using the upper bound $g(x_t)-g_{\opt} \leq C \sigma_t$ in \eqref{eq:bound-inner}, the definition of $z_t$ in \cref{z} and convexity of $g$, we have that
$$g(z_t)-g_{\opt} \leq \frac{C}{S_t}\left((t+1)t\sigma_t^2+\sum_{i \in [t]}(i+1)i(\sigma_{i-1}-\sigma_i)\sigma_i\right).$$
Since $S_t \geq (t+1)t\sigma_t$, we obtain
\[ g(z_t) - g_{\opt} \leq \frac{C}{(t+1)t\sigma_t} \left((t+1)t\sigma_t^2+ \sum_{i \in [t]}(i+1)i(\sigma_{i-1}-\sigma_i)\sigma_i\right). \]
The bound $g(z_t) - g_{\opt} \leq C(1+V) \sigma_t$ now follows from the definition of $V$.
\end{proof}

Next, we suggest parameter choices of $\{\sigma_t\}_{t \geq 0}$ satisfying all required conditions. Given a chosen meta-parameter $p \in (0,1)$, \cref{theorem: a convergence rate for PCG-BiO} then establishes convergence rates of $O(t^{-p})$ and $O(t^{-(1-p)})$ for the inner- and outer-level objectives, respectively.

\begin{lemma} \label{lemma: example of lambda}
    Given $p \in (0,1)$ and $\varsigma > 0$, the sequence $\sigma_t := \varsigma (t+1)^{-p}$ for $t \geq 0$ satisfies \crefrange{IRCG-C1}{IRCG-C3} with $L=p$. Furthermore, it holds that
    \[C \leq (1+2p)\left(f_{\opt}-\min_{x \in X} f(x)\right) + \frac{2(\varsigma L_f +L_g)D^2}{\varsigma}, \quad V \leq \frac{2p}{\min\{1,2(1-p)\}}\]
    where $C$ and $V$ are defined in \cref{eq:CV}.
\end{lemma}
\begin{proof} [Proof of \cref{lemma: example of lambda}]
It is clear that the sequence $\{\sigma_t\}_{t \geq 0}$ satisfies Conditions \ref{IRCG-C1} and \ref{IRCG-C2}.
To validate \cref{IRCG-C3}, we have that
\begin{align*}
\lim_{t \to \infty} t\left(\frac{\sigma_t}{\sigma_{t+1}}-1\right)=\lim_{t \to \infty} t\left(\left(1+\frac{1}{t+1}\right)^p-1\right) = \lim_{t \to \infty} \frac{t}{t+1}\cdot \lim_{\delta\to 0}\frac{(1+\delta)^p-1}{\delta}=p
\end{align*}
where the last equality holds because the derivative of $f(x)=x^p$ at $x=1$ equals $p$.

For $t>1$, it follows from the mean value theorem that %
there exists $b_t \in (t,t+1)$ 
such that
$$t^{-p}-(t+1)^{-p} = -pb_t^{-(p+1)}(t-t-1) = pb_t^{-(p+1)} \leq pt^{-(p+1)}.$$
Hence, we observe that
\begin{align*}
\sum_{i \in [t]} (i+1)i(\sigma_{i-1}-\sigma_i)&\leq 2 \varsigma p \sum_{i \in [t]} i^{1-p} \leq {2 \varsigma p}t^{2-p}
\end{align*}
where the first inequality is due to $t+1\leq 2t$ and the second inequality holds because $t^{1-p}$ is an increasing function in $t$. This implies 
\[ \frac{\sum_{i \in [t]} (i+1)i(\sigma_{i-1}-\sigma_i)}{(t+1)t\sigma_t} \leq \frac{2p t^{2-p}}{t(t+1)^{1-p}} \leq 2p. \]
Since $\min_{t \geq 0} (t+1)\sigma_t = \varsigma$, we obtain
$$C \leq (1+2p)\left(f_{\opt}-\min_{x \in X} f(x)\right) + \frac{2(\varsigma L_f +L_g)D^2}{\varsigma}.$$
Using similar arguments, we observe that
\begin{align*}
 \sum_{i \in [t]}(i+1)i(\sigma_{i-1}-\sigma_i)\sigma_i \leq 2\varsigma^2 p \sum_{i \in [t]} i^{1-2p}.
\end{align*}
If $1-2p \geq 0$, then $t^{1-2p}$ is an increasing function in $t$ and thus
$$\sum_{i \in [t]}(i+1)i(\sigma_{i-1}-\sigma_i)\sigma_i \leq 2p \varsigma^2 t^{2(1-p)}.$$
Dividing both sides of this inequality by $t(t+1) \sigma_t^2$, we deduce that $V \leq 2p$. When $1-2p <0$, we have 
\begin{align*}
\sum_{i \in [t]}(i+1)i(\sigma_{i-1}-\sigma_i)\sigma_i \leq 2p\varsigma^2\left(1+\frac{1}{2-2p}\left(t^{2-2p} -1\right)\right)\leq  \varsigma^2\frac{p}{1-p} t^{2(1-p)}
\end{align*}
where the first inequality is deduced from the integral approximation of the sum $\sum_{i \in [t]} i^{1-2p} \leq 1 + \int_{i=1}^t i^{1-2p} \mathrm{d}i$.
Dividing both sides of this inequality by $t(t+1) \sigma_t^2$, we deduce that $V \leq {p}/({1-p})$. In both cases, $V \leq {2p}/{\min\{1,2(1-p)\}}$.
\end{proof}

\begin{theorem} \label{theorem: a convergence rate for PCG-BiO}
Let $\{z_t\}_{t \geq 1}$ be the iterates generated by \cref{alg:IR-CG} with step sizes $\{\alpha_t\}_{t \geq 0}$ chosen via any one of \cref{eq:openloop,eq:exactlinesearch,eq:inexactlinesearch}, and regularization parameters $\{\sigma_t\}_{t \geq 0}$ given by $\sigma_t := \varsigma (t+1)^{-p}$ for some chosen $\varsigma > 0$, $p \in (0,1)$. Under \cref{ass:smooth-functions}, for any $t \geq 1$ we have
\begin{align*}
    f(z_t)-f_{\opt} &\leq \frac{2(\varsigma L_f+L_g)D^2}{\varsigma (t+1)^{1-p}} = O(t^{-(1-p)}),\\
    g(z_t) - g_{\opt} &\leq \frac{\varsigma(1+2p) \left(f_{\opt}-\min_{x \in X} f(x)\right) + 2(\varsigma L_f +L_g)D^2}{\frac{\min\left\{ 1,2(1-p) \right\}}{\min\{1+2p,2\}} \cdot (t+1)^p} = O(t^{-p}).
\end{align*}
\end{theorem}

\begin{proof}[Proof of \cref{theorem: a convergence rate for PCG-BiO}]
This is an immediate consequence of \cref{lemma:twoupperbounds,lemma: example of lambda}.
\end{proof}

	\section{Acceleration under quadratic growth and strong convexity}\label{sec:acceleration}

In this section, we show that the convergence rates of \cref{alg:IR-CG} provided in \cref{theorem: a convergence rate for PCG-BiO} can be improved under the following three additional assumptions.

\begin{assumption}[quadratic growth of $g$]\label{ass:g-quadratic-growth}
	There exists $\kappa > 0$ such that
	\[ \kappa \min_{y \in X_{\opt}}\|x-y\|^2 \leq g(x)-g_{\opt}, \quad \forall x \in X. \]
\end{assumption}

\cref{ass:g-quadratic-growth} (see also \citep{DrusvyatskiyLewis2018,KerdreuxEtAl2022}) is satisfied for a variety of functions; we present one example below.
\begin{example}\label{ex:accelerated}
Suppose $g(x):= x^\top A x-2b^\top x$ and $X:= \{x: \|x\|_2 \leq 1\}$ with $A$ is a positive semi-definite matrix, $b \in \mathrm{col}(A)$. If either $\lambda_{\min}(A) >0, \|A^\dagger b\|_2 \leq 1$, or $\lambda_{\min}(A)=0, \|A^\dagger b\|_2 < 1$, then \citet[Lemma 3.7]{RujunXudong2022} show that \cref{ass:g-quadratic-growth} holds.
\epr
\end{example}

The next two \crefrange{ass:strongly-convex-f}{ass:strongly-convex-X} are on strong convexity of $f$ and $X$. \citet{GarberHazan2015} examined CG methods under \crefrange{ass:strongly-convex-f}{ass:strongly-convex-X} in the single-level setting. \citet{Molinaro2023} provided an equivalence between sets satisfying \cref{ass:strongly-convex-X} and so-called \emph{gauge sets}, while \citet[Sec. 5.2]{WangEtAl2024} showed that it is possible to obtain acceleration for optimizing $f$ over gauge sets without \cref{ass:strongly-convex-f}. Note that the set $X$ in \cref{ex:accelerated} is an instance of a set that satisfies \cref{ass:strongly-convex-X} below.
\begin{assumption}[strong convexity of $f$]\label{ass:strongly-convex-f}
	There exists $\alpha_f >0$ such that for any $x,y \in X$, %
	$$f(y) \geq f(x)+\nabla f(x)^\top (y-x)+\frac{\alpha_f}{2}\|y-x\|^2.$$
\end{assumption}

\begin{assumption}[strong convexity of $X$]\label{ass:strongly-convex-X}
	There exists $\alpha_X >0$ such that for any $x,y \in X$, $u \in \mathbb{R}^n$ such that $\|u\|=1$, $\gamma \in [0,1]$, it holds that
	$$\gamma x +(1-\gamma) y + \frac{\gamma(1-\gamma)\alpha_X}{2}\|x-y\|^2 u \in X.$$
\end{assumption}

In this section, we will study \cref{alg:IR-CG} with parameters set to $\sigma_t = \varsigma(t+1)^{-p}$. \cref{theorem: a convergence rate for PCG-BiO} established a convergence rate of $O(t^{-p})$ and $O(t^{-(1-p)})$ respectively for inner- and outer-levels for the averaged sequence $z_t$. However, \cref{thm:accelerated-quadratic-growth} below shows that we can get a faster rate of $O(t^{-\min\{2p,1\}})$ for the inner-level under \cref{ass:g-quadratic-growth}, with a comparable $O\left( t^{-\min\{p,1-p\}} \right)$ outer-level rate. If we additionally impose \crefrange{ass:strongly-convex-f}{ass:strongly-convex-X} then \cref{thm:accelerated-strongly-convex} shows that we get $O\left(t^{-\min\{2/3,1-2p/3,2-2p\}}\right)$ for the outer-level gap. Furthermore, these guarantees now apply to $x_t$ rather than $z_t$. For both \cref{thm:accelerated-quadratic-growth,thm:accelerated-strongly-convex}, we justify the optimal choice of $p=1/2$ for the regularization parameter in \cref{rem:p-choice-accelerated} below. This results in $O(t^{-1})$ inner-level rate and $O(t^{-1/2})$ outer-level rate under \cref{ass:g-quadratic-growth}, which accelerates to $O(t^{-2/3})$ outer-level rate under \crefrange{ass:strongly-convex-f}{ass:strongly-convex-X}.

Our analysis exploits the following lemma. Recall that $\Phi_t(x) := \sigma_t f(x) + g(x)$.
\begin{lemma}\label{lem:fg-bounds-accelerated}
	Suppose that \crefrange{ass:smooth-functions}{ass:g-quadratic-growth} hold. Define $G_f := \max_{y \in X_{\opt}} \|\grad f(y)\|_*$, and let $\{x_t\}_{t \geq 0}$ be the sequence generated by \cref{alg:IR-CG}. Then
	\begin{align*}
		-\frac{G_f}{\sqrt{\kappa}} \sqrt{g(x_t) - g_{\opt}} \leq f(x_t) - f_{\opt} &\leq \frac{1}{\sigma_t} (\Phi_t(x_t) - \Phi_t(x_{\opt}))\\
		0 \leq g(x_t) - g_{\opt} &\leq 2(\Phi_t(x_t) - \Phi_t(x_{\opt})) + \frac{G_f^2 \sigma_t^2}{\kappa}.
	\end{align*}
\end{lemma}
\begin{proof}[Proof of \cref{lem:fg-bounds-accelerated}]
	Denote $f_t := f(x_t) - f_{\opt}$, $g_t := g(x_t) - g_{\opt}$ and $h_t := \Phi_t(x_t) - \Phi_t(x_{\opt}) = \sigma_t f_t + g_t$. For any $t \geq 0$, we define $y_t:= \argmin_{y \in X_{\opt}} \|y-x_t\|$. From convexity of $f$ and \cref{ass:g-quadratic-growth}, we have that
 \begin{align}
		f_t &\geq f(x_t)-f(y_t)\geq \grad f(y_t)^\top (x_t-y_t)\geq -\|\grad f(y_t)\|_* \|x_t-y_t\|\geq - \frac{G_f}{\sqrt{\kappa}} \sqrt{g_t}.\label{eq:fg-bound}
	\end{align}
	Therefore, $g_t - (G_f \sigma_t/\sqrt{\kappa})\sqrt{g_t}\leq h_t$. Denote $D_t := {G_f \sigma_t}/{\sqrt{\kappa}}$. Then we have
	\begin{align*}
		g_t - 2(D_t/2) \sqrt{g_t} &\leq h_t\quad 
		\iff \quad (\sqrt{g_t} - D_t/2)^2 \leq h_t + (D_t/2)^2.
	\end{align*}
	If $\sqrt{g_t} \leq D_t/2$, then $g_t \leq D_t^2/4$. Otherwise, the inequality becomes $\sqrt{g_t} \leq \sqrt{h_t + D_t^2/4} + D_t/2$. Squaring both sides and using the fact that $(a+b)^2 \leq 2a^2 + 2b^2$, we get $g_t \leq 2h_t + D_t^2$.
	Notice also that $x_t \in X$, thus $g_t \geq 0$. Therefore we have $f_t \leq h_t/\sigma_t$.
\end{proof}

\begin{theorem}\label{thm:accelerated-quadratic-growth}
Let $\{x_t,z_t\}_{t \geq 1}$ be the iterates generated by \cref{alg:IR-CG} with step sizes $\{\alpha_t\}_{t \geq 0}$ chosen via any one of \cref{eq:openloop,eq:exactlinesearch,eq:inexactlinesearch}, and regularization parameters $\{\sigma_t\}_{t \geq 0}$ given by $\sigma_t := \varsigma (t+1)^{-p}$ for some chosen $\varsigma > 0$, $p\in(0,1)$. Define $G_f := \max_{y \in X_{\opt}} \|\grad f(y)\|_*$. Under \crefrange{ass:smooth-functions}{ass:g-quadratic-growth}, for any $t \geq 1$ we have
\begin{align*}
		f(x_t)-f_{\opt} &\leq \frac{W}{2 \varsigma (t+1)^{\min\{p,1-p\}}} = O(t^{-\min\{p,1-p\}}), \\
   g(x_t) - g_{\opt} &\leq \frac{W}{(t+1)^{\min\{2p,1\}}} = O(t^{-\min\{2p,1\}}),
\end{align*}
where
\[ W := \max\left\{ \begin{aligned}
    &g(x_0) - g_{\opt}\\
    &\max_{w \in \bbR} \left\{ w : 4 (L_f \sigma_0 + L_g) D^2 + \frac{G_f^2 \varsigma^2}{\kappa} + \frac{p2^{\min\{2p+2,3\}}\sqrt{G_f^2 \varsigma^2 w}}{ \sqrt{\kappa}} \geq w\right\}
\end{aligned}\right\}. \]
\end{theorem}
\begin{proof}[Proof of \cref{thm:accelerated-quadratic-growth}]
Denote $f_t := f(x_t) - f_{\opt}$, $g_t := g(x_t) - g_{\opt}$ and $h_t := \Phi_t(x_t) - \Phi_t(x_{\opt}) = \sigma_t f_t + g_t$. \cref{lem:fg-bounds-accelerated} gives us that for $t \geq 1$, $f_t \geq - \sqrt{ {G_f^2 g_t}/\kappa }$. By \cref{ineq2:lemma:recursive-rule} in \cref{lemma:recursive-rule} we have for $t \geq 0$,
\[ h_{t+1} + (\sigma_t - \sigma_{t+1}) (f(x_{t+1}) - f_{\opt}) \leq \frac{t}{t+2} h_t + \frac{2 (L_f \sigma_t + L_g) D^2}{(t+2)^2}. \]
Combining with the above inequality, this implies
\[ h_{t+1} \leq \frac{t}{t+2} h_t + \frac{2 (L_f \sigma_t + L_g) D^2}{(t+2)^2} + \sqrt{\frac{G_f^2}{\kappa}} (\sigma_t - \sigma_{t+1}) \sqrt{g_t}. \]
Solving for this recursion (using the same technique as the proof of \cref{lemma: common bound}) gives the following bound on $h_t$:
\begin{equation}\label{eq:h-bound-accelerated}
h_t \leq \frac{2 (L_f \sigma_0 + L_g) D^2}{t+1} +  \frac{1}{t(t+1)} \sqrt{\frac{G_f^2} {\kappa}}\sum_{i \in [t]} i(i+1) (\sigma_{i-1} - \sigma_i) \sqrt{g_{i-1}}.
\end{equation}
Then \cref{lem:fg-bounds-accelerated} gives
\begin{align*}
g_t &\leq 2 h_t + \frac{G_f^2 \sigma_t^2}{\kappa}\\
&\leq \frac{4 (L_f \sigma_0 + L_g) D^2}{t+1} +  \frac{2}{t(t+1)} \sqrt{\frac{G_f^2} {\kappa}}\sum_{i \in [t]} i(i+1) (\sigma_{i-1} - \sigma_i) \sqrt{g_{i-1}} + \frac{G_f^2 \sigma_t^2}{\kappa}.
\end{align*}
Setting $\sigma_t = \varsigma(t+1)^{-p}$, we will show by induction that $g_t \leq {W}{(t+1)^{-\min\{2p,1\}}}$ for $t \geq 0$.
The base case $t=0$ follows since $W \geq g_0$. Now suppose that it holds for all $i \leq t$, for some $t \geq 0$. Then
\begin{align*}
	g_{t+1} &\leq \frac{4 (L_f \sigma_0 + L_g) D^2}{t+2} +  \frac{2\sqrt{G_f^2 /\kappa}}{(t+1)(t+2)} \sum_{i \in [t+1]} i(i+1) (\sigma_{i-1} - \sigma_i) \sqrt{g_{i-1}} + \frac{G_f^2 \sigma_{t+1}^2}{\kappa}\\
	&\leq \frac{4 (L_f \sigma_0 + L_g) D^2 + G_f^2 \varsigma^2/\kappa}{(t+2)^{\min\{2p,1\}}}\\
	&\quad +  \frac{2\sqrt{G_f^2 \varsigma^2 W/\kappa}}{(t+1)(t+2)} \sum_{i \in [t+1]} i(i+1) \left(\frac{1}{i^p} - \frac{1}{(i+1)^p}\right) \frac{1}{i^{\min \{p,1/2\}}}\\
	&\leq \frac{4 (L_f \sigma_0 + L_g) D^2 + G_f^2 \varsigma^2/\kappa}{(t+2)^{\min\{2p,1\}}} +  \frac{4p\sqrt{G_f^2 \varsigma^2 W/\kappa}}{(t+1)(t+2)} \sum_{i \in [t+1]} \frac{1}{i^{\min\{2p-1,p-1/2,0\}}}\\
	&\leq \frac{4 (L_f \sigma_0 + L_g) D^2 + G_f^2 \varsigma^2/\kappa}{(t+2)^{\min\{2p,1\}}} +  \frac{4p\sqrt{G_f^2 \varsigma^2 W/\kappa}}{(t+1)^{\min\{2p,p+1/2,1\}}}\\
 &= \frac{4 (L_f \sigma_0 + L_g) D^2 + G_f^2 \varsigma^2/\kappa}{(t+2)^{\min\{2p,1\}}} +  \frac{4p\sqrt{G_f^2 \varsigma^2 W/\kappa}}{(t+1)^{\min\{2p,1\}}}\\
 &\leq  \frac{4 (L_f \sigma_0 + L_g) D^2 + G_f^2 \varsigma^2/\kappa}{(t+2)^{\min\{2p,1\}}} +  \frac{p2^{\min\{2p+2,3\}}\sqrt{G_f^2 \varsigma^2 W/\kappa}}{(t+2)^{\min\{2p,1\}}}
\end{align*}
where the third inequality holds because $(i+1)\leq 2i$ for any $i\geq 1$, the mean value theorem implies that $i^{-p}-(i+1)^{-p}\leq p i^{-(p+1)}$ for any $i\geq 1$, and $\min\{2p-1,p-1/2,0\}\leq \min\{2p-1,p-1/2\}$.

Here, the final term is at most $W/(t+2)$ since we chose $W$ to satisfy
\[ W \geq \max_{w \in \bbR} \left\{ w : 4 (L_f \sigma_0 + L_g) D^2 + \frac{G_f^2 \varsigma^2}{\kappa} + \frac{p2^{\min\{2p+2,3\}}\sqrt{G_f^2 \varsigma^2 w}}{ \sqrt{\kappa}} \geq w\right\}. \]
Thus by induction $g_{t+1} \leq W(t+2)^{-\min\{2p,1\}}$ for all $t \geq 0$. Substituting this bound into \cref{eq:h-bound-accelerated}, using similar reasoning to the above we get $$h_t \leq \frac{2 (L_f \sigma_0 + L_g) D^2}{(t+2)^{\min\{2p,1\}}} +  \frac{p2^{\min\{2p+1,2\}}\sqrt{G_f^2 \varsigma^2 W/\kappa}}{(t+2)^{\min\{2p,1\}}}.$$ 
By the definition of $W$, the right hand side is at most $W/(2(t+2)^{\min\{2p,1\}})$. Therefore $f_t \leq h_t/\sigma_t \leq W/(2 \varsigma (t+2)^{\min\{p,1-p\}})$, as required.
\end{proof}

We now examine how \crefrange{ass:strongly-convex-f}{ass:strongly-convex-X} can additionally be used to obtain a faster rate for $f$. The analysis relies on the following technical lemma to solve a recursion.
\begin{lemma}\label{lem:h-recursion}
	Let $\{b_t,h_t,\sigma_t\}_{t \geq 1}$ be two sequences of real numbers that satisfy
	\[ h_{t+1} \leq h_t\cdot \max\left\{ \frac{1}{2}, \ 1-M \sqrt{\sigma_t [h_t]_+} \right\} + b_t \]
	for some $M > 0$. Suppose that there exist constants $H,\tau,k$ such that $h_{\tau} \leq H/(\tau+2)^k$, and for all $t \geq \tau$, we have
	\[ \frac{H}{2(t+2)^k} + b_t \leq \frac{H}{(t+3)^k}, \quad  \frac{H k}{(t+2)^{k+1}} + b_t \leq \frac{MH^{3/2}\sqrt{\sigma_t}}{\sqrt{2} (t+2)^{3k/2}}. \]
	Then
	\[ h_t \leq \frac{H}{(t+2)^k}, \quad \forall t \geq \tau. \]
\end{lemma}
\begin{proof}[Proof of \cref{lem:h-recursion}]
We proceed by induction
and assume that $h_t \leq H/(t+2)^k$ up to some $t \geq \tau$. If ${1}/{2} \geq 1- M \sqrt{\sigma_t [h_t]_+}$, then 
	\[ h_{t+1} \leq \frac{h_t}{2}+b_t \leq \frac{H}{2(t+2)^k} + b_t \leq \frac{H}{(t+3)^k}. \]
	If $1/2< 1- M \sqrt{\sigma_t [h_t]_+}$, then we consider two subcases.
	
	\emph{Case 1.} If $h_t \leq H/(2(t+2)^{k})$, then $h_{t+1} \leq h_t+b_t \leq {H}/{(t+3)^k}$.
	
	\emph{Case 2.} If $h_t > H/(2(t+2)^{k})$, then
	\begin{align*}
		h_{t+1} &\leq \frac{H}{(t+2)^k} \left(1- \frac{M\sqrt{H \sigma_t}}{\sqrt{2(t+2)^{k}}} \right)+b_t= \frac{H}{(t+2)^k} +b_t - \frac{MH^{3/2}\sqrt{\sigma_t}}{\sqrt{2} (t+2)^{3k/2}}.
	\end{align*}
	For any $t \geq 0$, by the mean value theorem, there exists $c_t \in (t+2,t+3)$ such that
	$$\frac{1}{(t+2)^k}-\frac{1}{(t+3)^k} = \frac{k}{c_t^{k+1}} \leq\frac{k}{(t+2)^{k+1}}.$$
Then it follows that
	\begin{align*}
		h_{t+1} &\leq \frac{H}{(t+3)^k} + \frac{H k}{(t+2)^{k+1}} +b_t - \frac{MH^{3/2}\sqrt{\sigma_t}}{\sqrt{2} (t+2)^{3k/2}} \leq \frac{H}{(t+3)^k}.
	\end{align*}
	Then the result holds by induction.
\end{proof}

\begin{theorem}\label{thm:accelerated-strongly-convex}
Consider the same conditions as in \cref{thm:accelerated-quadratic-growth}. Additionally, suppose that \crefrange{ass:strongly-convex-f}{ass:strongly-convex-X} hold, and that $\{\alpha_t\}_{t \geq 0}$ in \cref{alg:IR-CG} are chosen via \cref{eq:inexactlinesearch} or \cref{eq:exactlinesearch}. Then for $t \geq 1$,
\begin{align*}
    f(x_t) - f_{\opt} &\leq \frac{H}{\varsigma (t+2)^{\min\{2/3,1-2p/3,2-2p\}}} = O(t^{-\min\{2/3,1-2p/3,2-2p\}}),\\ 
    g(x_t) - g_{\opt} &\leq \frac{W}{(t+1)^{\min\{2p,1\}}} = O(t^{-\min\{2p,1\}}),
\end{align*}
where $W$ is as defined in \cref{thm:accelerated-quadratic-growth}, and
\begin{align*}
    M &:= \frac{\alpha_X \sqrt{\alpha_f}}{8\sqrt{2}(\sigma_0 L_f + L_f)},\\
    k&:= \min\{p+2/3,p/3+1,2-p\}\\
    H &:= \min_{h \geq 0} \left\{ h : \forall t \geq 1, \begin{aligned}
    \varsigma 3^{k} (\Phi_1(x_1) - \Phi_1(x_{\opt})) &\leq h,\\
	\frac{h}{2(t+2)^k} + \frac{\sqrt{G_f W \varsigma^2} p2^{p+1}}{\sqrt{\kappa}(t+2)^{\min\{2p+1,p+3/2\}}} &\leq \frac{h}{(t+3)^k},\\
\frac{h k}{(t+2)^{k+1}} + \frac{\sqrt{G_f W \varsigma^2}p2^{p+1}}{\sqrt{\kappa}(t+2)^{\min\{2p+1,p+3/2\}}} &\leq \frac{Mh^{3/2} \sqrt{\sigma_t}}{\sqrt{2} (t+2)^{3k/2+p/4}}
\end{aligned} \right\}.
\end{align*}
Furthermore, the constant $H$ as defined above is finite.
\end{theorem}
\begin{proof}[Proof of \cref{thm:accelerated-strongly-convex}]
Denote $f_t := f(x_t) - f_{\opt}$, $g_t := g(x_t) - g_{\opt}$ and $h_t := \Phi_t(x_t) - \Phi_t(x_{\opt}) = \sigma_t f_t + g_t$. From \cref{lemma:recursive-rule}, since $\alpha_t$ is chosen via \cref{eq:inexactlinesearch} or \cref{eq:exactlinesearch}, for any $\alpha \in [0,1]$ we have
\[ h_{t+1}+(\sigma_t-\sigma_{t+1})f_{t+1} \leq h_t+\alpha \grad \Phi_t(x_t)^\top (v_t-x_t)+\frac{(\sigma_t L_f+L_g)\alpha^2}{2}\|v_t-x_t\|^2.
\]
Since $x_t,v_t \in X$,
we have from \cref{ass:strongly-convex-X} that for any $u \in \mathbb{R}^n$ with $\|u\|=1$,
\[ \frac{1}{2}(x_t+v_t)+\frac{\alpha_X}{8}\|v_t-x_t\|^2 u \in X.\]
By our choice of $v_t$, we deduce that
\begin{align*}
	\nabla \Phi_t(x_t)^\top (v_t-x_t) &\leq \nabla \Phi_t(x_t)^\top \left(\frac{1}{2}(x_t+v_t)+\frac{\alpha_X}{8}\|v_t-x_t\|^2 u-x_t\right)  \\
	&=\frac{1}{2}\nabla \Phi_t(x_t)^\top (v_t-x_t)+\frac{\alpha_X\|v_t-x_t\|^2}{8} u^\top \nabla \Phi_t(x_t).
\end{align*}
By setting $u$ such that $u^\top \nabla \Phi_t(x_t)= -\|\nabla \Phi_t(x_t)\|_*$,  we obtain
\begin{align*}
	\nabla \Phi_t(x_t)^\top (v_t-x_t) &\leq  \frac{1}{2}\nabla \Phi_t(x_t)^\top (x_{\opt}-x_t)-\frac{\alpha_X\|v_t-x_t\|^2}{8} \|\nabla \Phi_t(x_t)\|_*\\
	&\leq -\frac{1}{2}h_t-\frac{\alpha_X\|v_t-x_t\|^2}{8} \|\nabla \Phi_t(x_t)\|_*
\end{align*}
where the first inequality is by our choice of $v_t$ and the second inequality is from convexity of $\Phi_t(\cdot)$.
Combining with the previous inequality, it holds that
\begin{align*}
	&h_{t+1}+(\sigma_t-\sigma_{t+1})f_{t+1}\\
	&\leq \left(1-\frac{\alpha}{2}\right)h_t+\frac{\|v_t-x_t\|^2}{2}\left((\sigma_tL_f+L_g)\alpha^2-\frac{\alpha_X\|\grad \Phi_t(x_t)\|_*}{4}\alpha
	\right).
\end{align*}
If $\alpha_X\|\nabla \Phi_t(x_t)\|_* \geq 4(\sigma_tL_f+L_g)$, then by setting $\alpha = 1$, we have $$h_{t+1}+(\sigma_t-\sigma_{t+1})f_{t+1} \leq \frac{h_t}{2}.$$ 
Otherwise, we substitute $\alpha = \alpha_X\|\nabla \Phi_t(x_t)\|_*/ 4(\sigma_tL_f+L_g)$ and obtain
\begin{align*}
	h_{t+1}+(\sigma_t-\sigma_{t+1})f_{t+1} &\leq h_t\left(1- \frac{\alpha_X\|\nabla \Phi_t(x_t)\|_*}{8(\sigma_tL_f+L_g)}\right)\leq h_t\left(1- \frac{\alpha_X\|\nabla \Phi_t(x_t)\|_*}{8(\sigma_0L_f+L_g)}\right).
\end{align*}
In both cases, we have
\begin{align*}
	h_{t+1}+(\sigma_t-\sigma_{t+1})f_{t+1} \leq h_t \cdot\max\left\{\frac{1}{2},\ 1- \frac{\alpha_X\|\nabla \Phi_t(x_t)\|_*}{8(\sigma_0L_f+L_g)}\right\},
\end{align*}
For each $t \geq 0$, we define $w_t := \argmin_{x \in X} \Phi_t(x)$ and thus by the optimality condition $\grad \Phi_t(w_t)^\top(x_t-w_t) \geq 0$. \cref{ass:strongly-convex-f} then gives
\[ \Phi_t(x_t)-\Phi_t(w_t) \geq \grad \Phi_t(w_t)^\top(x_t-w_t)+ \frac{\alpha_f \sigma_t}{2}\|x_t-w_t\|^2 \geq \frac{\alpha_f \sigma_t}{2}\|x_t-w_t\|^2. \]
Thus, we obtain
\begin{align*}
	\left(\frac{2}{\alpha_f \sigma_t}\left(\Phi_t(x_t)-\Phi_t(w_t)\right)\right)^{1/2} \|\nabla \Phi_t(x_t)\|_* 
	&\geq \|x_t-w_t\| \|\nabla \Phi_t(x_t)\|_*\\
	&\geq \grad \Phi_t(x_t)^\top (x_t-w_t) \\
	&\geq \Phi_t(x_t)-\Phi_t(w_t)\\
	\implies \|\nabla \Phi_t(x)\|_* \geq \left(\frac{\alpha_f \sigma_t}{2}\left(\Phi_t(x_t)-\Phi_t(w_t)\right)\right)^{1/2} &\geq \left(\frac{\alpha_f \sigma_t}{2}[h_t]_+\right)^{1/2}
\end{align*}
where the last inequality is a result of $\Phi_t(x_t)-\Phi_t(w_t) \geq \Phi_t(x_t)-\Phi_t(x_{\opt})$ and $\Phi_t(x_t)-\Phi_t(w_t) \geq 0$.

Thus, we obtain
$$h_{t+1}+(\sigma_t-\sigma_{t+1}) f_{t+1} \leq h_t \cdot\max\left\{\frac{1}{2},\ 1- \frac{\alpha_X}{8(\sigma_0L_f+L_g)}\left(\frac{\alpha_f \sigma_t}{2}[h_t]_+\right)^{1/2}\right\}.$$
Let us define $b_t := -(\sigma_t - \sigma_{t+1}) f_{t+1}$. We thus have
\[ h_{t+1} \leq h_t \cdot \max\left\{\frac{1}{2},\ 1- M \sqrt{ \sigma_t [h_t]_+} \right\} + b_t. \]

We now seek to apply \cref{lem:h-recursion} so that $h_t \leq H/(t+2)^k$ for $t \geq 1$, assuming the existence of $H,k$ satisfying the conditions of the lemma. We show that $k=\min\{p+2/3,p/3+1,2-p\}$ and $H$ as defined above will satisfy the conditions. Furthermore, we will argue that $H$ is finite. Combining the bounds in \cref{lem:fg-bounds-accelerated,thm:accelerated-quadratic-growth} gives $f_{t+1} \geq - \sqrt{\frac{G_f W}{\kappa(t+2)^{\min\{1,2p\}}}}$.
As $\sigma_t-\sigma_{t+1}= \varsigma((t+1)^{-p} - (t+2)^{-p})$, we have
\begin{align*}
	b_t &= -(\sigma_t - \sigma_{t+1}) f_{t+1} \leq \sqrt{\frac{G_f W \varsigma^2}{\kappa}} \frac{p(t+1)^{-p-1}}{(t+2)^{\min\{p,1/2\}}}\\
	&\leq \sqrt{\frac{G_f W \varsigma^2}{4\kappa}} \frac{p2^{p+1}(t+2)^{-p-1}}{(t+2)^{\min\{p,1/2\}}} \leq \sqrt{\frac{G_f W \varsigma^2}{\kappa}} \frac{p2^{p+1}}{(t+2)^{\min\{2p+1,p+3/2\}}}.
\end{align*}
Using this bound on $b_t$, the two conditions required for $H,k$ in \cref{lem:h-recursion} are:
\begin{align*}
	\frac{H}{2(t+2)^k} + \sqrt{\frac{G_f W \varsigma^2}{\kappa}} \frac{p2^{p+1}}{(t+2)^{\min\{2p+1,p+3/2\}}} &\leq \frac{H}{(t+3)^k},\\
	\frac{H k}{(t+2)^{k+1}} + \sqrt{\frac{G_f W \varsigma^2}{\kappa}} \frac{p2^{p+1}}{(t+2)^{\min\{2p+1,p+3/2\}}} &\leq \frac{MH^{3/2} \sqrt{\sigma_t}}{\sqrt{2} (t+2)^{3k/2}}.
\end{align*}
We choose $k=\min\{p+2/3,p/3+1,2-p\}$, which means that the conditions become
\begin{align*}
	\frac{H}{2(t+2)^k} + \sqrt{\frac{G_f W \varsigma^2}{\kappa}} \frac{p2^{p+1}}{(t+2)^{\min\{2p+1,p+3/2\}}} &\leq \frac{H}{(t+3)^k},\\
	\frac{H k}{(t+2)^{k+1}} + \sqrt{\frac{G_f W \varsigma^2}{\kappa}} \frac{p2^{p+1}}{(t+2)^{\min\{2p+1,p+3/2\}}} &\leq \frac{MH^{3/2} \sqrt{\sigma_t}}{\sqrt{2} (t+2)^{3k/2+p/4}}.
\end{align*}
We note that the value we chose for $k$ satisfies
$$k< 3k/2+p/2 \leq \min\{2p+1,p+3/2,k+1\}.$$
Since the left hand sides decrease at least as fast as the right hand sides, we can choose $H$ sufficiently large so that these inequalities will be satisfied for all $t \geq 1$, thus the conditions of \cref{lem:h-recursion} are satisfied and we deduce that $h_t \leq H/(t+2)^{k}$.

Finally, using the fact that $\sigma_t = \varsigma (t+1)^{-p}$, \cref{lem:fg-bounds-accelerated} now gives $f_t \leq h_t/\sigma_t = (H/\varsigma)/(t+2)^{\min\{2/3,1-2p/3,2-2p\}}$.
\end{proof}

\begin{remark}\label{rem:p-choice-accelerated}
For both \cref{thm:accelerated-quadratic-growth,thm:accelerated-strongly-convex}, the optimal choice of the regularization parameter is $p=1/2$. To see this, notice that the rate for the inner-level is $O\left(t^{-\min\{2p,1\}}\right)$ in both theorems, and that $\min\{2p,1\} \leq 1$, with equality when $p=1/2$. For \cref{thm:accelerated-quadratic-growth}, the rate for the outer-level is $O\left(t^{-\min\{p,1-p\}}\right)$, and $\min\{p,1-p\} \leq 1/2$, with equality when $p=1/2$. For \cref{thm:accelerated-strongly-convex}, the rate for the outer-level is $O\left(t^{-\min\{p,1-p\}}\right)$, and $\min\{2/3,1-2p/3,2-2p\} \leq 2/3$, with equality when $p=1/2$. In summary, for both theorems, the fastest rates achievable for both inner- and outer-level optimality gaps occur when $p=1/2$.
\epr
\end{remark}

	\section{Numerical study} \label{sec:experiments}

We perform numerical experiments on matrix completion \citep{LuEtAl2021}, which seeks to find a low-rank $n \times p$ matrix $X$ to approximate a subset of observed entries $M_{i,j}$, for $(i,j) \in \Omega \subset [n] \times [p]$. This is done by solving the following optimization problem:
\begin{equation} \label{single level matrix completion}
\begin{aligned}
\min_{X \in \mathbb{R}^{n \times p}} \quad g(X):=\frac{1}{2} \sum_{(i,j) \in \Omega}(X_{i,j}-M_{i,j})^2\quad
\text{s.t.} \quad \|X\|_* \leq \delta,
\end{aligned}
\end{equation}
where $\|\cdot\|_*$ is the nuclear norm and $\delta$ is a positive constant. The objective of \eqref{single level matrix completion} is not strictly convex, so it is possible to have multiple minimizers. %
One possible criterion to select between different minima is to choose one with the lowest variance within columns, leading us to the following outer-level objective:
\begin{equation} \label{eq:sum-of-squares outer}
    f(X)=\frac{1}{2} \sum_{j \in [p]} \sum_{i \in [n]}\left(X_{i,j}-\overline{X}_j \right)^2\quad\text{where}\quad \overline{X}_j := \frac{1}{n}\sum_{i \in [n]}X_{i,j}, \ \forall j \in [p].
\end{equation}
Since $\left(\overline{X}_1,\dots, \overline{X}_p\right) = X^\top \mathbf{1}_{n}/n$, 
we can rewrite $f$ as
\begin{equation} \label{eq:outer-objective}
    f(X) = \frac{1}{2}\left\|U X \right\|_F^2\quad\text{where}\quad U:=\mathbf{I}_n-\frac{1}{n}\mathbf{1}_{n}\mathbf{1}_{n}^\top
\end{equation}
and $\|\cdot\|_F$ is the Frobenius norm.
To this end, we solve the following bilevel problem:
\begin{equation} \label{bilevel level matrix completion}
\begin{aligned}
\min_{X \in \mathbb{R}^{n \times p}} \quad  f(X)\quad
\text{s.t.} \quad  X \in \argmin_{\|Z\|_* \leq \delta} g(Z),
\end{aligned}
\end{equation}
From \eqref{eq:sum-of-squares outer}, $f$ is a convex quadratic function in terms of $X$. Furthermore, since $U = U^\top U$, $U$ is positive semi-definite and has the largest eigenvalue of $1$. Thus, the smoothness constants for $f$ and $g$ are $L_f = L_g =1$.

\subsection{Data description}
We use the MovieLens 1M data set \citep{grouplens}. This data set contains ratings of $3952$ movies from $6040$ users, made on a 5-star scale. Therefore, $n=6040$, $p=3952$, and each $M_{i,j} \in [5]$ for $(i,j) \in \Omega$. In this context, the objective $f$ above looks for matrices $X$ where ratings of particular movies across users have low variance. In the dataset, we have $|\Omega| = 1,000,209$ observed entries, which is $\approx 4.19\%$ of total possible entries. In our experiments, we set the nuclear norm radius to be $\delta = 5$. 

\subsection{Algorithms}
We implemented our \ircg method given by \cref{alg:IR-CG} to solve \eqref{bilevel level matrix completion}. For \ircg, we implemented the three step size selection schemes as described in \cref{eq:openloop,eq:exactlinesearch,eq:inexactlinesearch}. For regularization parameters, we select $\sigma_t = 0.05(t+1)^{-1/2}$ for each $t \geq 0$ with $p=1/2$, which ensures inner- and outer-level objectives converge at rate $O(t^{-1/2})$. We implemented the line search \cref{eq:exactlinesearch} via the bounded Brent method from the \texttt{scipy.optimize.minimize-scalar} package (version 1.11.3).

Since $f(X+\alpha \mathbf{1}_{n}\mathbf{1}_{p}^\top) = f(X)$ for any $X \in \bbR^{n \times p}$ and $\alpha \in \mathbb{R}$, the sublevel sets of $f$ are not compact. Hence, a critical assumption %
for the ITALEX method \citep{DoronShtern2022} is violated. As a result, for performance comparison, we only implemented \cgbio \citep{JiangEtAl2023}, \irpg \citep{Solodov2006}, and \bisg\citep{MerchavSabach2023}. We chose the parameters for the implementation of these algorithms based on the
criteria described in the corresponding papers. Specifically, for \cgbio, we set the step sizes to be $\alpha_t = 2/(t+2)$ for $t \geq 0$ and $\epsilon_g = 10^{-4}$. Following the notation in the original papers: for \irpg, we set $\theta = \Tilde{\alpha} = \eta = 1/2$ and regularization parameters $\sigma_t = 0.05(t+1)^{-1/2}$ for each $t \geq 0$; for \bisg, we set $\alpha = 1/(2-0.01)$ (to get the convergence rates of both inner- and outer-level objectives close to $O\left(t^{-1/2}\right)$), and $c=\min\{1/L_f,1\}=1$.

The starting points for all algorithms except for \cgbio were set to be the following matrix:
$$X_0 := 0.01 \times \delta \begin{bmatrix}
            \mathbf{I}_p/p & \mathbf{0}_{p\times (n-p)}
        \end{bmatrix}^\top.$$
For \cgbio, we required an initial point $X'_0$ that satisfies $g(X'_0)-g_{\opt} \leq \epsilon_g/2$. To generate this point, we used a single-level conditional gradient algorithm for the inner-level objective with step sizes $\alpha_t = 2/(t+2)$ for $t \geq 0$ until the surrogate gap was $\leq \epsilon_g/2$. We initialized this phase with $X_0$ given above.

For \ircg, we had to solve linear minimization sub-problems over the nuclear norm ball, whose solution is discussed in Appendix \ref{subsec:NNB}. For \cgbio, we tackled a linear oracle over the nuclear norm ball intersecting with a half-space, and the corresponding solution is provided in Appendix~\ref{subsec:SNNB}. To compute the projection onto the feasible set required for \irpg and \bisg, we followed the steps given in Appendix~\ref{subsec:projectionNNB}.

To approximate the inner-level optimal value $g_{\opt}$, we implemented conditional gradient starting from $X_0$ with step sizes $\alpha_t = 2/(t+2)$ for $t \geq 0$ to retrieve a $10^{-5}$-sub-optimal solution using duality surrogate gap as the stopping criterion. Then we used this sub-optimal solution as a starting point for the implementation of another conditional gradient run to obtain a $10^{-12}$-sub-optimal solution and $g_{\opt}$ was approximated by the corresponding inner-level objective value. We found that this warm-up approximation scheme saved time significantly compared to running the algorithm only once.

We set a time limit of $10$ minutes ($600$ seconds) for all algorithms. All experiments were run on a server with a 2.4GHz processor and 32 GB memory, using Python 3.10.9.

\begin{figure}[t]
\centering

\includegraphics{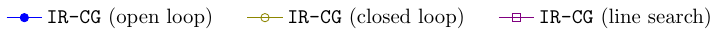}
\includegraphics{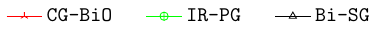}
\includegraphics{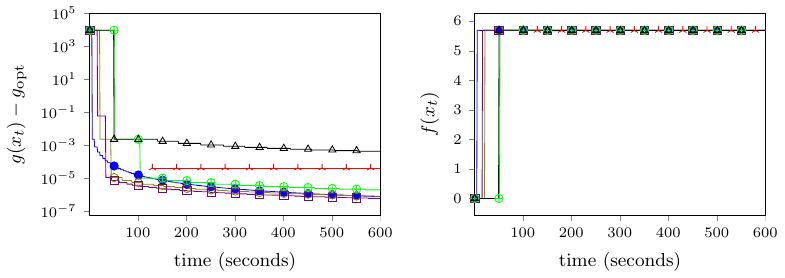}

\caption{Plot of the best inner-level objective value found by each algorithm (left) and the corresponding outer-level objective value (right) on the low-rank matrix completion instance, at each point in time. 
Note that the $y$-axis is in logarithmic scale on the left figure. In addition, for all versions of \cref{alg:IR-CG}, we report the values based on the last iterate $x_t$, which converges according to \cref{theorem:asymptotic}.
}
\label{fig:plot2}
\end{figure}

\begin{table}[h!t]
\centering
\begin{tabular}{||c || c||}
\hline
 Method & Number of iterations executed\\
 \hline
 \hline
 \ircg \ (open loop)  & 110  \\
 \ircg \ (closed loop)  & 52  \\
 \ircg \ (line search) & 33\\
 \cgbio & 2\\
 \irpg   & 12\\
 \bisg & 13\\
 \hline
\end{tabular}
\caption{Comparison of the number of iterations executed within $10$ minutes.}
\label{table:1}
\end{table}

\subsection{Results comparison}

\cref{fig:plot2} illustrates the values of the inner-level optimality gap (on the left) and outer objective (on the right) generated by \ircg, \cgbio, \irpg, \bisg within $10$ minutes. Regarding the inner-level optimality gap, we observe that the three versions of \ircg perform comparably best, followed by \irpg. According to \cref{table:1}, \irpg ran only $12$ iterations within the time limit while \ircg (open loop) ran 110 iterations. 
\cgbio barely makes any improvement compared to the initialized point $X'_0$ since it can only perform $2$ iterations due to the complicated structure of the linear minimization oracle. Although \bisg is known to have a theoretical convergence rate of $O(t^{-1/(2-0.01)})$ for the inner-level objective in this particular problem class, %
it shows an inferior performance %
compared to \ircg. \cref{fig:plot2} also highlights that the outer-level objective values of these algorithms are directly correlated to the inner-level optimality gaps.

	\section{Conclusion} \label{sec:conclusion}

In this paper, we provide a projection-free algorithm for %
convex bilevel optimization 
via the iterative regularization approach. Through this approach, we obtain rates of convergence in terms of both inner- and outer-level objective values simultaneously, as well as accelerated rates under additional assumptions.

One possible future research direction is to
explore projection-free bilevel methods that mitigate the well-known zig-zag behavior of objective values under conditional gradient-type methods, which often results in slow local convergence.
Another direction is to consider stochastic algorithms, which may be useful in large-scale problems in which exact gradient computation may be expensive.

        \backmatter
        
        \bmhead{Acknowledgements}
        
        The authors would like to thank the review team for feedback and suggestions which significantly improved the paper.

	\begin{appendices}

\section{Implementation details}\label{sec:appendix-implementation}

\subsection{Linear minimization over a nuclear norm ball.} \label{subsec:NNB}
We consider the linear sub-problem for implementing \ircg. %
The corresponding linear sub-problem is
\begin{equation} \label{MC:linearsub1st}
\begin{aligned}
\min_{V} \quad \Trace(C^\top V)\quad
\text{s.t.} \quad \|V\|_* \leq \delta.
\end{aligned}
\end{equation}
Let $u^{(1)}, v^{(1)}$ be the left and right leading singular vectors of $C$. Then \citet[Section 4.2]{Jaggi2013} suggests that the solution of \eqref{MC:linearsub1st} is 
$V^*:= -\delta u^{(1)}(v^{(1)})^\top.$
To compute this solution, we compute the leading eigenvector $v^{(1)}$ with length $1$ and the largest eigenvalue $\sigma_{\max}^2(C)$ of $C^\top C$ with the Lanczos process \citep[Section 10.1]{GolubEtAl2013} via package \texttt{scipy.linalg.eigh} (version 1.11.3). Vector $u^{(1)}$ is computed as $Cv^{(1)}/\sigma_{\max}(C)$.

\subsection{Linear minimization over a sliced nuclear norm ball.} \label{subsec:SNNB}
For \cgbio \citep{JiangEtAl2023}, the corresponding linear sub-problem is as follows:
\begin{equation} \label{MC:linearsub2nd}
\min_{V} \quad \Trace(C^\top V)\quad\text{s.t.} \quad \|V\|_* \leq \delta,\ \Trace(A^\top V) \leq b.
\end{equation}
Since the size of $V$ is large, it is impractical to use off-the-shelf conic optimisation solvers for \eqref{MC:linearsub2nd}. Therefore we provide an efficient custom algorithm. We note that we will consider the case when \eqref{MC:linearsub2nd} is feasible, since they are generated from outer approximations of $X_{\opt}$, which we assume to be non-empty. This implies that $b \geq \min_{\|V\|_* \leq \delta} \{\Trace(A^\top V)\} = -\delta\sigma_{\max}(A)$.
Before continuing, we have the following observation.
\begin{lemma} \label{lemma:Slater}
If 
    $b >  -\delta\sigma_{\max}(A),$
    then Slater's condition holds for problem \eqref{MC:linearsub2nd}. If $A \neq 0$, then the reverse is true.
\end{lemma}
\begin{proof}[Proof of \cref{lemma:Slater}]
If 
    $b >  -\delta\sigma_{\max}(A) = \min_{\|V\|_* \leq \delta} \{\Trace(A^\top V)\},$
    we let $V^*$ be a minimizer  of $\Trace(A^\top V)$ over $\|V\|_* \leq \delta$  such that $\|V^*\|_* = \delta$ and observe that there exists a sufficiently small $\epsilon>0$ such that for any $U$ such that $\|U-V^*\|_* < \epsilon$, $\Trace(A^\top U) < b$ by the continuity of linear function. Since $V^*$ is on the boundary of the nuclear norm ball, there must exist $U$ that satisfies $\|U-V^*\|_* < \epsilon$ and is in the interior of the nuclear norm ball, i.e., $\|U\|_* < \delta$. Thus, Slater's condition holds for problem \eqref{MC:linearsub2nd}.
    Now, we assume $A \neq 0$. If Slater's condition holds for problem \eqref{MC:linearsub2nd}, there exists $U$ such that
    $\|U\|_* < \delta$ and $\Trace(A^\top U) \leq b.$
     Since $A \neq 0$ and $\|U\|_*<\delta$, we have that
    $b \geq \Trace(A^\top U) \geq -\sigma_{\max}(A)\|U\|_* > -\delta\sigma_{\max}(A)$,
    as required.
\end{proof}

First, we consider the case in which Slater's condition does not hold. In this case, by \cref{lemma:Slater}, we have $\min_{\|V\|_* \leq \delta}\{\Trace(A^\top V)\} = -\delta \sigma_{\max}(A) = b,$ and
$$\{V \in \mathbb{R}^{n \times p} \mid \Trace(A^T V) \geq b, \|V\|_* \leq \delta\} = \argmin_{\|U\|_* \leq \delta} \left\{\Trace(A^\top U)\right\}.$$
Therefore, a solution to \cref{MC:linearsub2nd} in this case is
$$V^* \in \argmin_{V} \left\{\Trace(C^\top V) \;\middle|\; V \in \argmin_{\|U\|_* \leq \delta} \left\{\Trace(A^\top U)\right\}\right\}.$$
When Slater's condition holds for problem \eqref{MC:linearsub2nd}, we consider the Lagrangian
$$L(V,\lambda, \mu) = \Trace((C+\lambda A)^\top V)-\lambda b + \mu\left(\|V\|_* - \delta\right), \quad V \in \mathbb{R}^{n \times p}, \lambda, \mu \geq 0,$$
then the dual function is
$$\mathcal{D}(\lambda, \mu) = \inf_{V \in \mathbb{R}^{n \times p}}L(V,\lambda, \mu) = \begin{cases}-b\lambda-\delta \mu, & \sigma_{\max}(C+\lambda A) \leq \mu\\ \ -\infty, & \text{otherwise}. \end{cases}$$
Thus, the dual problem is
\begin{equation}\label{1stdualprob}
    \begin{aligned}
        \max_{\lambda, \mu \geq 0} \quad -b\lambda-\delta \mu\quad
        \text{s.t.} \quad \sigma_{\max}(C+\lambda A) \leq \mu,
    \end{aligned}
\end{equation}
which is equivalent to
\begin{equation}\label{2nddualprob}
    \min_{\lambda \geq 0} \quad \delta \sigma_{\max}(C+\lambda A) + b\lambda.
\end{equation}
We also note that in case Slater's condition holds, we have $b>-\delta\sigma_{\max}(A)$ if $A \neq 0$ by \cref{lemma:Slater}. At any optimal solution $\lambda^*$ to problem \eqref{2nddualprob}, the objective must not be greater than that at $\lambda = 0$. Therefore, we can obtain an upper bound on $\lambda^*$ as follows:
$$ \delta \sigma_{\max}(C) \geq \delta \sigma_{\max}(C+\lambda^* A) + b\lambda^*
                   \geq \delta \left(\sigma_{\max}(\lambda^* A)-\sigma_{\max}(-C)\right)+ b\lambda^*$$
where we use the triangle inequality for the spectral norm in the second inequality. This is equivalent to ${2\delta \sigma_{\max}(C)}/({b+\delta\sigma_{\max}(A)}) \geq \lambda^*$.

 Given an optimal dual variable $\lambda^*$ by solving \eqref{2nddualprob}, the optimal solution to problem \eqref{MC:linearsub2nd} is also the solution of minimizing $\Trace((C+\lambda^* A)^\top V)$ over the nuclear ball. If $A = 0$, we observe that $\lambda^* = 0$ minimizes problem \eqref{2nddualprob} since $b \geq -\delta \sigma_{\max}(A) = 0$. 
\begin{remark} \label{remark:optimal-dual}
To solve problem \eqref{2nddualprob}, if $A \neq 0$, we compute $\lambda^*$ by conducting line search on function $\delta \sigma_{\max}(C+\lambda A) + b\lambda$ over the interval 
$\left[0,{2\delta \sigma_{\max}(C)}/({b+\delta\sigma_{\max}(A)})\right]$, 
with the bounded Brent method via package \texttt{scipy.optimize.minimize-scalar} (version 1.11.3).
If $A=0$, we set $\lambda^*:=0$. 
    \epr
\end{remark}
Given $\lambda^*$, we need to ensure the solution we get from minimizing $\Trace((C+\lambda^* A)^\top V)$ over the nuclear ball satisfies the linear inequality constraint. Hence, we can compute such solution as follows:
$$V^*\in \argmin_{V} \left\{\Trace(A^\top V) \;\middle|\; V \in \argmin_{\|U\|_* \leq \delta} \left\{\Trace((C+\lambda^* A)^\top U)\right\}\right\}.$$
Therefore, both cases require us to solve bilevel linear problems over the nuclear norm ball. To do this, we need the following results.

\begin{lemma} \label{lemma:argminNB}
Given a matrix $P \in \mathbb{R}^{n \times p}$, let $E_1$ be the eigenspace associated with the leading eigenvalue of matrix $P^\top P$ and
$$\mathcal{E}_1:= \left\{ -\frac{\delta}{\sigma_{\max}(P)} P v v^\top \;\middle|\; v \in E_1, \ \|v\|_2=1\right\}.$$
Then 
$$\argmin_{\|Z\|_*\leq \delta} \Trace(P^\top Z) =  \Conv(\mathcal{E}_1).$$
\end{lemma}
\begin{proof}[Proof of \cref{lemma:argminNB}]
Note that the optimal value is $-\delta \sigma_{\max}(P)$. First, we prove that $\argmin_{\|Z\|_*\leq \delta} \Trace(P^\top Z) \supseteq  \Conv(\mathcal{E}_1)$ holds.
Given $X \in \mathcal{E}_1$, there exists $v \in E_1$ with $\|v\|_2=1$ such that
$X=-(\delta/\sigma_{\max}(P)) P v v^\top.$
Observe that $X^\top X = ({\delta^2}/{\sigma^2_{\max}(P)}) vv^\top P^\top Pvv^\top = \delta^2 vv^\top$. Note that there is only one non-zero eigenvalue $\delta^2$ of $X^\top X$ with eigenvector $v$. Any vector that is orthogonal to $v$ has eigenvalue $0$. Therefore $(X^\top X)^{1/2} = \delta vv^\top$, and $\|X\|_* = \Tr((X^\top X)^{1/2}) = \delta$, thus $X$ is feasible.
This implies that any point in $\Conv(\mathcal{E}_1)$ is also feasible. We have that
$$\Trace(P^\top X) = -\frac{\delta}{\sigma_{\max}(P)} \Trace(v^\top P^\top Pv)=-\delta \sigma_{\max}(P),$$
thus $X$ is optimal, which implies $\mathcal{E}_1 \subseteq \argmin_{\|Z\|_*\leq \delta} \Trace(P^\top Z)$. Taking convex hull for both sets, we obtain the required result.

Now, we prove $\argmin_{\|Z\|_*\leq \delta} \Trace(P^\top Z) \subseteq  \Conv(\mathcal{E}_1)$ holds. Given a feasible $X$, let a singular value decomposition of $X$ be $X = \sum_{i\in [\min\{n,p\}]} \sigma_i u_i v_i^\top$,
where $\sigma_1 \geq \dots \geq \sigma_{\min\{n,p\}} \geq 0$. Since $\|X\|_* \leq \delta$, we have
$\sum_{i\in [\min\{n,p\}]} \sigma_i \leq \delta$.

By using the Cauchy-Schwarz inequality and the fact that $\{u_i\}_{i}$ and $\{v_i\}_{i}$ are two sets orthonormal vectors, we have that
$$ \Trace(P^\top X)=\sum_{i\in [\min\{n,p\}]}\sigma_i v_i^\top P^\top u_i\geq \sum_{i\in [\min\{n,p\}]}\sigma_i\left( -\|u_i\|_2 \|P v_i\|_2\right).$$
Here, the inequality becomes equality if and only if given $i \in [\min\{n,p\}]$, we have that $\sigma_i = 0$ or $Pv_i= k_iu_i$ for some $k_i \leq 0$. Since $\|u_i\|_2=1$,
$$ \Trace(P^\top X)\geq -\sum_{i\in [\min\{n,p\}]} \sigma_i \|Pv_i\|_2\geq - \sigma_{\max}(P) \sum_{i\in [\min\{n,p\}]} \sigma_i\geq - \delta \sigma_{\max}(P),$$
where the second inequality becomes equality if and only if given $i \in [\min\{n,p\}]$, $\sigma_i =0$ or $v_i$ is a leading eigenvector of $P^\top P$. Hence, in order to have 
$X \in \argmin_{\|Z\|_*\leq \delta} \Trace(P^\top Z)$,
given $i \in [\min\{n,p\}]$, we have that $\sigma_i =0$ or $v_i$ is a normalized leading eigenvector of $P^\top P$.
In summary, in order for $X$ to be optimal (i.e., all inequalities above hold with equality), we need $v_i$ to be a leading eigenvector of $P^\top P$, $Pv_i = -\sigma_{\max}(P) u_i$ whenever $\sigma_i \neq 0$, and $\sum_{i \in [\min\{n,p\}]} \sigma_i = \delta$. Furthermore, note that if we define $u_i = -({1}/\sigma_{\max}(P))P v_i$ whenever $v_i \in E_1$, then $u_i^\top u_j = ({1}/{\sigma_{\max}(P)^2}) v_i^\top P^\top P v_j = v_i^\top v_j = 0$ for any other $j \in [\min\{n,p\}]$. On the other hand, $u_i^\top u_i = v_i^\top v_i = 1$. Therefore $\{u_i\}_{i\in [\min\{n,p\}]}$ defined in this way is also an orthonormal set of vectors.

Since $\{v_i\}_{i \in [\min\{n,p\}]}$ is an orthonormal set of vectors in $E_1$, we must have
$\sigma_i = 0, \forall i> \dim(E_1)$. Therefore $X = \sum_{i\in [\dim(E_1)]} \sigma_i u_i v_i^\top = -({1}/{\sigma_{\max}(P)}) \sum_{i\in [\dim(E_1)]} \sigma_i P v_i v_i^\top$,
where $\{v_1,\dots, v_{\dim(E_1)}\}$ is an orthonormal basis of $E_1$ and $\sigma_i \geq 0, \sum_{i\in [\dim(E_1)]}\sigma_i = \delta$. Therefore, $X \in \Conv(\mathcal{E}_1)$ as required.
\end{proof}

\begin{lemma} \label{lemma:bilevelNB}
Given $Q \in \mathbb{R}^{n \times p}$, let $P, E_1$ be defined as in \cref{lemma:argminNB}, $R \in \mathbb{R}^{p \times \dim(E_1)}$ be a matrix whose columns form an orthonormal basis of $E_1$,  $S \in \mathbb{R}^{\dim(E_1) \times \dim(E_1)}$ be a symmetric matrix defined as $S:= R^\top\left(( Q^\top P +P^\top Q)/{2}\right)R$,
and $s_1 \in \mathbb{R}^{d_1}$ be a leading eigenvector of $S$ of length $1$. Then we have that  
$$-\frac{\delta}{\sigma_{\max}(P)}  P (Rs_1) (Rs_1)^\top \in \argmin_{X}\left\{ \Trace(Q^\top X) 
 \;\middle|\; X \in \argmin_{\|Z\|_*\leq \delta} \left\{\Trace(P^\top Z)\right\}\right\}.$$
\end{lemma}

\begin{proof}[Proof of \cref{lemma:bilevelNB}]
First, from \cref{lemma:argminNB}, we observe that
$$\argmin_{X}\left\{ \Trace(Q^\top X) 
 \;\middle|\; X \in \argmin_{\|Z\|_*\leq \delta} \left\{\Trace(P^\top Z)\right\}\right\} \supseteq 
 \argmin_{ X \in \mathcal{E}_1}\left\{ \Trace(Q^\top X) \right\}.$$
    Let $X \in \mathcal{E}_1$. Then there exists $v \in E_1$ with  $\|v\|_2=1$ such that 
   $X=-(\delta/\sigma_{\max}(P))P v v^\top.$
Then we have
\begin{align*}
    \Trace(Q^\top X) &= -\frac{\delta}{\sigma_{\max}(P) } v^\top (Q^\top P)v \geq -\frac{\delta}{\sigma_{\max}(P)} \max_{\|u\|_2 = 1, u \in E_1} \{u^\top (Q^\top P)u\}.
\end{align*}
Such lower bound can be obtained when we set $v \in \argmax_{\|u\|_2 = 1, u \in E_1} \{u^\top (Q^\top P)u\}$. Given $\|u\|_2 = 1$ with $u \in E_1$, we have $u = Rs$, where $s \in \mathbb{R}^{\dim(E_1)}$ is a vector of length $1$. Hence, we have
$$\argmax_{\|u\|_2 = 1, u \in E_1} \{u^\top (Q^\top P)u\} = \argmax_{\|s\|_2 = 1}\{s^\top (R^\top Q^\top P R)s\}=\argmax_{\|s\|_2 = 1}\{s^\top S s\}.$$
Thus, we can choose $v = Rs_1$.
\end{proof}

\begin{remark} \label{remark:compute-R}
To compute $R$ as defined in \cref{lemma:bilevelNB}, we used package \texttt{scipy.linalg.eigh} (version 1.11.3) to compute the leading eigenvalue of matrix $P^\top P$ and the associated eigenvectors whose lengths are $1$.
\epr
\end{remark}
Now, we have enough tools to address problem \eqref{MC:linearsub2nd}, which are shown in \cref{alg:NBBLO} and \cref{alg:SNBLO}. 

\begin{algorithm}[htpt]
\caption{Bilevel linear oracle over nuclear norm ball - $\nbblo(P,Q,\delta)$} \label{alg:NBBLO}
\KwData{$P, Q \in \mathbb{R}^{n \times p}, \delta>0$.}
\KwResult{$V^* \in \argmin_{V} \left\{\Trace(Q^\top V) \mid V \in \argmin_{\|U\|_* \leq \delta} \left\{\Trace(P^\top U)\right\}\right\}$.}
 Compute 
        \vspace{-5pt}
        \begin{align*}
            R & \quad \textnormal{as outlined in \cref{remark:compute-R}}, \qquad & S &:= \frac{1}{2}R^\top(Q^\top P +P^\top Q)R &&&&\\
            s_1 & \in \argmax_{\|s\|_2 = 1} \{s^\top S s\}, \qquad & V^* &:= \frac{-\delta}{\sigma_{\max}(P)}P (Rs_1) (Rs_1)^\top.&&&&
        \end{align*}
        \vspace{-5pt}
\end{algorithm}

\begin{algorithm}[htpt]
\caption{Linear oracle over nuclear sliced norm ball - \snblo} \label{alg:SNBLO}
\KwData{$C \in \mathbb{R}^{n \times p}, A \in \mathbb{R}^{n \times p}, b \in \mathbb{R}, \delta>0$.}
\KwResult{$V^*$- a solution of \eqref{MC:linearsub2nd}.}
\uIf{$b = -\delta\sigma_{\max}(A)$}{Compute 
\begin{equation*}
    V^* := \nbblo(A,C,\delta).
\end{equation*}
}
\Else{  Compute 
\vspace{-5pt}
\begin{align*}
    \lambda^* &\in \argmin_{\lambda \geq 0} \{\delta \sigma_{\max}(C+\lambda A) + b\lambda\}, \quad V^* := \nbblo (C+\lambda^* A, A,\delta).
\end{align*}
        \vspace{-5pt}}
\end{algorithm}
\subsection{Projection onto a nuclear norm ball.} \label{subsec:projectionNNB}
Given a matrix $X$ and its singular value decomposition as 
$X= \sum_{i\in [k]} \sigma_i u_i v_i^\top,$
in which $k = \min\{n,p\}$ and $\sigma_1 \geq \dots \geq \sigma_k \geq 0$. Let $s \in \mathbb{R}^k$ be the Euclidean projection of $(\sigma_1, \dots,\sigma_k)$ onto the set $S_{\delta}:=\{x \in \mathbb{R}^k \mid \mathbf{1}^\top x \leq \delta, x \geq 0\}$. \citet[Section 7.3.2]{Beck2017} shows that the Frobenius norm projection of $X$ onto the nuclear ball $\{V \in \mathbb{R}^{n \times p} \mid \|V\|_* \leq \delta\}$ is
$\sum_{i\in [k]} s_i u_i v_i^\top.$
Now we discuss how $s$, the projection of $(\sigma_1,\dots,\sigma_k)$ onto $S_{\delta}$, can be computed. When $\delta = 1$, \citet[Algorithm 1]{CondatLaurent2016} provides an efficient method for computing the projection onto $S_1$. For general $\delta > 0$, the projection onto $S_{\delta}$ can be computed as $\Proj_{S_\delta}(x) = \delta \Proj_{S_1}\left(x/\delta\right)$.

\subsection{Verification of \ref{ass:f-smooth} for \eqref{eq:outer-objective}}
From the definition of $f$, $f$ is a convex function over $\mathbb{R}^{n \times p}$. We define 
$P:=I_{n}-\mathbf{1}_{n}\mathbf{1}_{n}^\top/n$,
and observe that $P^\top P=P$.
We also have that $Px = x$ for any $x \in \mathbb{R}^n$ such that $x \perp \mathbf{1}_n$. Therefore, the largest singular value of $P$ is $1$.
Since $\nabla F(X) = P^\top PX = PX$ for any $X \in \mathbb{R}^{n \times p}$, we have that for any $X,Y \in \mathbb{R}^{n \times p}$,
$\|\nabla f(X)-\nabla f(Y)\|_F =  \left\| P (X-Y)\right\|_F$.
Let $z_1,\dots,z_p$ be the columns of $(X-Y)$. We observe that 
$\|P(X-Y)\|_F^2 = \sum_{j \in [p]}\|Pz_j\|_2^2\leq \sum_{j \in [p]} \|z_j\|_2^2 =  \|X-Y\|_F^2$.
Therefore, we have $\|\nabla f(X)-\nabla f(Y)\|_F  \leq \|X-Y\|_F$.

	\end{appendices}


\end{document}